\documentclass[a4paper,reqno,twoside]{amsart}
\usepackage[left=3.5cm,right=3.5cm,top=3.5cm,bottom=3.5cm,footskip=1.5cm,headsep=1cm,bindingoffset=0.cm,]{geometry}
\usepackage{hyperref}
\hypersetup{hidelinks}
\usepackage{orcidlink}
\usepackage{amssymb}
\usepackage{bm}
\usepackage{cleveref}
\usepackage{cite}
\crefname{equation}{}{}
\crefname{figure}{Figure}{Figures}
\crefname{table}{Table}{Tables}
\crefname{section}{Section}{Sections}
\crefname{theorem}{Theorem}{Theorems}
\crefname{proposition}{Proposition}{Propositions}
\crefname{lemma}{Lemma}{Lemmas}
\crefname{corollary}{Corollary}{Corollaries}

\newtheorem{theorem}{Theorem}[section]
\newtheorem{proposition}[theorem]{Proposition}
\newtheorem{lemma}[theorem]{Lemma}

\theoremstyle{definition}
\newtheorem{definition}[theorem]{Definition}

\theoremstyle{remark}
\newtheorem{remark}[theorem]{Remark}

\numberwithin{equation}{section}

\begin{document}

\title[A frequency-stable Nystr\"om method]{A frequency-stable Nystr\"om method for two-dimensional diffraction by periodic gratings at and near Wood anomalies}

\author[K. Matsushima]{Kei Matsushima\,\orcidlink{0000-0002-0352-8770}}
\address{\parbox{\linewidth}{Kei Matsushima\\
  Informatics and Data Science Program, Graduate School of Advanced Science and Engineering / International Institute for Sustainability with Knotted Chiral Meta Matter (WPI-SKCM2), Hiroshima University, Higashi-Hiroshima, Japan, \href{https://orcid.org/0000-0002-0352-8770}{orcid.org/0000-0002-0352-8770}}}
  \email{matsushima@acs.hiroshima-u.ac.jp}

\author[H. Isakari]{Hiroshi Isakari\,\orcidlink{0000-0002-7906-5477}}
\address{\parbox{\linewidth}{Hiroshi Isakari\\
  Department of System Design Engineering, Faculty of Science and Technology, Keio University, Yokohama, Japan, \href{https://orcid.org/0000-0002-7906-5477}{orcid.org/0000-0002-7906-5477}}}

\begin{abstract}
  We propose and analyze a frequency-stable boundary integral equation method for scalar wave scattering by periodic diffraction gratings in two dimensions. The standard quasi-periodic Green function is singular as a function of the frequency at Wood anomaly frequencies, which prevents a direct frequency-uniform error analysis of conventional integral equation discretizations. Our method is based on a modified quasi-periodic Green function obtained by subtracting the singular Rayleigh modes and introducing a finite-dimensional correction that restores the radiation condition. The resulting combined-field boundary integral equation is shown to be uniquely solvable whenever the underlying diffraction problem is uniquely solvable, and the corresponding scattered field depends continuously on the frequency, including at Wood anomalies. We then construct a fully discrete Nystr\"om scheme using an Ewald representation of the Green function and a quadrature rule adapted to the logarithmic singularity of the kernel. For compact frequency intervals that may contain Wood anomaly frequencies, we prove unique solvability of the discrete system for sufficiently large discretization and truncation parameters and establish a frequency-uniform error estimate and bound on the condition number. Numerical examples illustrate the convergence of the scheme and its stability across Wood anomaly frequencies.
\end{abstract}

\maketitle

\section{Introduction}
Let us consider a periodic curve $\Gamma_{\rm per}$ in $\mathbb R^2$. Precisely, let $L>0$ and let $\gamma:\mathbb R\to\mathbb R^2$ be an embedding of class $C^{1,1}$ such that $\gamma(t+2\pi)=\gamma(t)+Le_1$, where $e_1:=(1,0)$, and $|\gamma^\prime(t)|>0$ for all $t\in\mathbb R$. We define $\Gamma_\mathrm{per} := \gamma(\mathbb R)$ and its one-period segment $\Gamma:=\gamma([0,2\pi])$. 
We denote by $\Omega_+$ the connected component of $\mathbb R^2\setminus\Gamma_\mathrm{per}$ for which there exists $H_+ \in \mathbb R$ such that $\mathbb R \times (H_+,\infty) \subset \Omega_+$, and by $\Omega_-$ the component for which there exists $H_- \in \mathbb R$ such that $\mathbb R \times (-\infty,H_-) \subset \Omega_-$.

Let $\beta\in\mathbb R$ and let $C_\beta(\Gamma_\mathrm{per}):=\{ \varphi\in C(\Gamma_\mathrm{per}) : \varphi(x+Le_1)=\varphi(x)\mathrm e^{\mathrm i\beta} \}$ be the space of continuous $\beta$-quasi-periodic functions on $\Gamma_\mathrm{per}$ equipped with $\| \varphi\|_{C_\beta(\Gamma_\mathrm{per})} := \|\varphi\|_{C(\Gamma)}$. For given $k\in\mathbb R$ and $g\in C_\beta(\Gamma_\mathrm{per})$, we consider 

\begin{align*}
    (P_k) \quad \text{Find }&u\in C^2(\Omega_+)\cap C(\overline\Omega_+) \text{ such that }
    \\&
    \begin{cases}
        -\varDelta u - k^2 u = 0 &\text{in }\Omega_+,
        \\
        u(x+Le_1) = u(x)\mathrm e^{\mathrm i\beta} &\text{in }\Omega_+,
        \\
        u = g &\text{on }\Gamma_\mathrm{per},
        \\
        u \text{ satisfies the upward radiation condition,}
    \end{cases}
\end{align*}
where the definition of the upward radiation condition is given in the next section. 

The problem $(P_k)$ models the diffraction of scalar waves by periodic surfaces in two dimensions. Such diffraction problems arise, for example, in modeling of diffraction gratings and metasurfaces. Diffraction gratings are classical optical elements used in spectral analysis and laser systems \cite{bonod2016diffraction}, whereas metasurfaces are designed to control scattering amplitude, phase, and polarization and to shape optical wavefronts \cite{chen2016review}.

This type of diffraction problem, which couples a Helmholtz-type equation with a linear boundary condition, a quasi-periodicity condition, and a radiation condition, has a well-established mathematical theory. For instance, a diffraction problem for Maxwell's system is uniquely solvable except for a discrete set of frequencies \cite{nedelec1991integral}. For scalar diffraction problems with Dirichlet boundary conditions, an overhang-free geometry provides a standard sufficient condition for uniqueness of the homogeneous problem \cite{kirsch1993diffraction}. These issues are closely related to spectral theory for periodic or quasi-periodic differential operators, as well as to the possible occurrence of guided modes \cite{alber1979quasi,wilcox1984scattering,bonnetbendhia1994guided}.

In view of the practical importance of diffraction analyses, we are interested in numerical approximation of solutions to diffraction problems. A promising approach is to formulate the diffraction problem on a bounded computational domain using perfectly matched layers or transparent boundary conditions, and then discretize the resulting weak formulation by finite elements \cite{chen2003adaptive,wang2015adaptive}. Another approach is based on boundary integral equations. Green-function formulations of grating diffraction lead to boundary integral equations on the periodic surface \cite{berg1971diffraction}. Many numerical schemes have been developed based on boundary integral equations with the quasi-periodic Green function, e.g., \cite{arens2006integral,rathsfeld2006fast,bruno2009efficient,pinto2021fast,OTANI20084630}. The direct use of the quasi-periodic Green functions is, however, not always valid due to the Wood anomaly. 

Several complementary boundary-integral strategies have been developed for computations at and around Wood anomalies. Shifted quasi-periodic Green functions replace the classical lattice sum by finite differences of shifted free-space kernels. Since the corresponding spectral representation may omit a finite set of Rayleigh modes, suitable plane-wave completion terms are added to obtain full-spectrum formulations \cite{bruno2014rapidly,bruno2017three}. A closely related hybrid spatial/spectral viewpoint explicitly separates the Rayleigh terms that become singular at a Wood anomaly from the regular part of the quasi-periodic Green function. The singular contributions then appear as finite-rank operators and are treated using Woodbury--Sherman--Morrison-type formulas together with a limiting procedure \cite{BrunoFernandezLado2017,BrunoFernandezLado2020}. These approaches provide well-conditioned full-spectrum solvers for a range of periodic scattering configurations. Another class of methods avoids the quasi-periodic Green function and instead uses the free-space fundamental solution together with auxiliary unknowns or proxy sources to impose quasi-periodicity \cite{barnett2011new,cho2015robust,wu2023robust}. More recently, a corrected windowed Green-function formulation was introduced in which a finite-rank modal correction enforces the Rayleigh radiation condition; the resulting equations can be discretized by standard Nyström or boundary element methods and produce superalgebraically convergent numerical solutions throughout the spectrum \cite{StrauszerCaussadeEtAl2023}. A related PML-BIE formulation incorporating a finite-mode correction has also recently been proposed
\cite{TanPerezArancibiaYin2026}.

A remaining issue is to establish rigorous error estimates that are uniform with respect to the frequency, in particular near Wood anomaly frequencies. For related two-dimensional periodic scattering problems, superalgebraic convergence has been proved for a spectral Galerkin method for multilayered periodic media away from Wood anomaly frequencies \cite{pinto2021fast} and for a Nystr\"om method applied to integral equations arising in scattering by diffraction gratings and rough surfaces under smoothness assumptions \cite{meier2000nystrom}. What is still missing, in particular, is a stability and error analysis that is locally uniform with respect to the frequency across Wood anomaly frequencies. Because of the Wood anomalies, we cannot immediately conclude that the numerical error can be bounded by a constant that is (locally) uniform in frequency. Indeed, we cannot expect such uniform error bounds for the standard boundary integral equation method with the quasi-periodic Green function due to the singular dependence of the kernel on the frequency at Wood anomaly frequencies, and the resulting loss of uniform boundedness around such frequencies.

In this paper, we discuss a finite rank-augmented boundary integral formulation for the diffraction problem $(P_k)$ and its Nystr\"om discretization. Our formulation is based on the observation that a non-radiating quasi-periodic Green function can still be constructed at Wood anomaly frequencies. In the present two-dimensional setting, this Green function is given explicitly by a Rayleigh expansion with a linear term in the grazing modes; see \cref{prop:green-modified}. The three-dimensional counterpart was developed in \cite{bruno2017three} and was called \textit{all-space (non-radiating) Green's function}. However, we cannot immediately obtain a boundary integral equation from this all-space Green function owing to two difficulties: (i) the all-space Green function does not satisfy the radiation condition and (ii) the usual radiating quasi-periodic Green function does not converge to this non-radiating Green function as the frequency approaches to a Wood anomaly. In this paper, we tackle these challenges through a careful analysis of quasi-periodic and all-space Green functions. These difficulties (i) and (ii) can be removed by subtracting the singular Rayleigh-mode contributions from the quasi-periodic Green function and adding a finite-dimensional correction to restore the radiation condition. Subsequently, the integral equation is discretized using the Nystr\"om method. In view of the logarithmic singularity of the kernel, we develop a special quadrature rule and its error analysis. Combining these results, we establish a rigorous error estimate for the proposed fully discrete numerical scheme and a frequency-uniform bound on the condition number of the discrete system matrix on compact frequency intervals that may contain Wood anomaly frequencies.

The rest of this paper is organized as follows. In Section 2, we introduce the radiation condition and quasi-periodic Green functions, formulate the boundary integral equation, give a complete specification of the fully discrete Nystr\"om scheme, and state the main results, \cref{thm:main-1,thm:main-2}. Section 3 is devoted to the analysis of the Green function, including its continuity with respect to the frequency and its Ewald representation. Section 4 describes the boundary integral formulation and proves its well-posedness. In Section 5, we derive the Nystr\"om discretization and prove frequency-uniform bounds for both the discretization error and the condition number of the resulting system matrix. Finally, we present numerical examples of the proposed scheme in Section 6.

\section{Preliminaries and main results}

In what follows, we always fix $L>0$ and $\beta\in\mathbb R$. Without loss of generality, we assume that the unit normal vector field $\nu:\Gamma_\mathrm{per}\to\mathbb R^2$, given by $\nu(\gamma(t)) := (-\gamma^\prime_2(t),\gamma^\prime_1(t))/|\gamma^\prime(t)|$, points into $\Omega_+$ (i.e., outward with respect to $\Omega_-$). Throughout the paper, we take the square root of a real number to have nonnegative imaginary part, i.e., $\sqrt{x}\ge 0$ for $x\ge 0$ and $\sqrt{x}=\mathrm i\sqrt{-x}$ for $x<0$. We begin with some preliminaries to describe our results.

\subsection{Radiation conditions and uniqueness}
Formally, we say that a solution of the Helmholtz equation is radiating if it can be written as a superposition of waves propagating in outward directions. To make this precise, let $\Omega\subset\mathbb R^2$ be an open and unbounded subset such that $\Omega+Le_1 = \Omega$. We first assume that $H_+ := \inf \{ h\in\mathbb R : \mathbb R \times (h,\infty)\subset\Omega \} \in\mathbb R$. Let $u\in C^2(\Omega)$ satisfy $-\varDelta u - k^2 u = 0$ and $u(x+Le_1) = u(x)\mathrm e^{\mathrm i\beta}$ in $\Omega$. By elliptic regularity, $u(x)\mathrm e^{-\mathrm i \beta x_1/L}$ is real-analytic in $\mathbb R\times(H_+,\infty)$. Hence its Fourier series in the periodic variable $x_1$, together with the differentiated series, converges locally uniformly. Thus termwise differentiation shows that there exists a unique pair of sequences $(A^+_s)$ and $(B^+_s)$ in $\mathbb C$ such that 
\begin{align*}
    u(x) &= \sum_{s\in\mathbb Z\setminus \sigma_\ast(k)} A^+_s\mathrm e^{\mathrm ik_sx_1 + \mathrm i\sqrt{k^2-k_s^2}x_2} + \sum_{s\in\mathbb Z\setminus \sigma_\ast(k)} B^+_s\mathrm e^{\mathrm ik_sx_1 - \mathrm i\sqrt{k^2-k_s^2}x_2} 
    \\
    & + \sum_{s\in\sigma_\ast(k)} A^+_s \mathrm e^{\mathrm ik_sx_1} + \sum_{s\in\sigma_\ast(k)} B^+_s x_2 \mathrm e^{\mathrm ik_sx_1}
    \quad\text{in }\mathbb R\times(H_+,\infty),
\end{align*}
where $k_s:=(\beta+2s\pi)/L$ and $\sigma_\ast(k) := \{s\in\mathbb Z : k^2-k_s^2=0\}$ is called the Wood-anomaly index set. Note that $\#\sigma_\ast(k) \leq 2$ for all $k\in\mathbb R$. We say that the solution $u$ satisfies the upward radiation condition if $B^+_s = 0$ for all $s\in\mathbb Z$.

The same argument can be applied in the opposite direction. For $\Omega$ with $H_- := \sup \{ h\in\mathbb R : \mathbb R \times (-\infty,h)\subset\Omega \} \in\mathbb R$, we can construct sequences $(A^-_s)$ and $(B^-_s)$ in $\mathbb C$ such that 
\begin{align*}
    u(x) &= \sum_{s\in\mathbb Z\setminus \sigma_\ast(k)} A^-_s\mathrm e^{\mathrm ik_sx_1 - \mathrm i\sqrt{k^2-k_s^2}x_2} + \sum_{s\in\mathbb Z\setminus \sigma_\ast(k)} B^-_s\mathrm e^{\mathrm ik_sx_1 + \mathrm i\sqrt{k^2-k_s^2}x_2} 
    \\
    & + \sum_{s\in\sigma_\ast(k)} A^-_s \mathrm e^{\mathrm ik_sx_1} + \sum_{s\in\sigma_\ast(k)} B^-_s x_2 \mathrm e^{\mathrm ik_sx_1}
    \quad\text{in }\mathbb R\times(-\infty,H_-).
\end{align*}
We say that the solution $u$ satisfies the downward radiation condition if $B^-_s = 0$ for all $s\in\mathbb Z$.

A Rellich-type uniqueness argument can be applied to the homogeneous version of $(P_k)$ under a suitable geometric condition.

\begin{theorem}[Kirsch \cite{kirsch1993diffraction}]\label{thm:kirsch}
    Let $k\in\mathbb R$. Assume that $\nu_2(x)>0$ for all $x\in\Gamma_\mathrm{per}$.
    Then the problem $(P_k)$ with $g=0$ admits only the trivial solution $u=0$.
\end{theorem}
The condition $\nu_2>0$ means that $\Gamma_{\rm per}$ is an overhang-free periodic graph.

\subsection{Quasi-periodic Green function}
Quasi-periodic Green functions play a central role in boundary integral formulations of $(P_k)$. We first discuss their distributional definition, including the case of Wood anomaly frequencies.

Let $\mathcal D(\mathbb R^2) = C_c^\infty(\mathbb R^2;\mathbb C)$ be the space of compactly supported smooth complex-valued functions on $\mathbb R^2$ equipped with the usual locally convex topology. The space of distributions on $\mathbb R^2$, denoted by $\mathcal D^\prime(\mathbb R^2)$, is the continuous dual of $\mathcal D(\mathbb R^2)$. For example, the Dirac delta $\delta_x\in \mathcal D^\prime(\mathbb R^2)$ at $x\in\mathbb R^2$ is given by $\delta_x(\varphi) = \varphi(x)$ for all $\varphi\in \mathcal D(\mathbb R^2)$. It is clear that the lattice sum of Dirac deltas $F_\# := \sum_{n=-\infty}^\infty \delta_{nLe_1} \mathrm{e}^{\mathrm in\beta}$ converges in $\mathcal D^\prime(\mathbb R^2)$. Using the distribution $F_\#$, supported on the lattice $\Theta:=\{nLe_1 : n\in\mathbb Z\}$, we define quasi-periodic Green functions as follows.
\begin{definition}
    Let $k\in\mathbb R$. A distribution $G\in\mathcal D^\prime(\mathbb R^2)$ is called a quasi-periodic Green function if it satisfies $-\varDelta G - k^2G = F_\#$ in $\mathcal D^\prime(\mathbb R^2)$ and $G(x+Le_1) = G(x)\mathrm e^{\mathrm i\beta}$ for all $x\in\mathbb R^2\setminus\Theta$.
\end{definition}
In this definition, the pointwise quasi-periodic condition makes sense since any solution of the distributional Helmholtz equation can be identified as a smooth function by elliptic regularity away from the support of $F_\#$. More precisely, we obtain the following observation.

\begin{lemma}\label{lemma:green-regularity}
    Let $k\in\mathbb R$ and let $G\in\mathcal D^\prime(\mathbb R^2)$ be a solution to $-\varDelta G - k^2G = F_\# $ in $\mathcal D^\prime(\mathbb R^2)$. 
    Then $G\in L_\mathrm{loc}^p(\mathbb R^2)$ for all $p\in[1,\infty)$, and $G$ is real-analytic in $\mathbb R^2\setminus\Theta$.
    Moreover, for each $n\in\mathbb Z$ the difference $G-\Phi^{(n)}_k \mathrm e^{\mathrm in\beta} $ is real-analytic in a neighborhood of $nLe_1$, where 
    \begin{align*}
        \Phi^{(n)}_k(x) :=
        \begin{cases}
        \displaystyle
            \frac{1}{2\pi}\log \frac{1}{|x-nLe_1|} & \text{if }k=0
        \\
        \displaystyle
            \frac{\mathrm i}{4}H^{(1)}_0(|k||x-nLe_1|) & \text{otherwise}
        \end{cases}
        \quad\text{in }\mathbb R^2
    \end{align*}
    with the Hankel function $H^{(1)}_0$ of the first kind and zeroth order.
\end{lemma}
\begin{proof}
    The analyticity is an immediate consequence of elliptic regularity. 
    We see that $\Phi_n\in L_\mathrm{loc}^p(\mathbb R^2)$ for all $p\in[1,\infty)$ and $n\in\mathbb Z$. Moreover, each $\Phi^{(n)}_k$ solves $-(\varDelta + k^2)\Phi^{(n)}_k = \delta_{nLe_1}$ in the distributional sense. Therefore, for each $V\Subset\mathbb R^2$ we can choose sufficiently large $N\in\mathbb N$ such that $G - \sum_{n=-N}^N \Phi^{(n)}_k\mathrm e^{\mathrm in\beta}$ is real-analytic in $V$. Hence $G\in L_\mathrm{loc}^p(\mathbb R^2)$ for all $p\in[1,\infty)$. In particular, choosing $V$ small enough around $nLe_1$, we obtain that $G-\Phi^{(n)}_k \mathrm e^{\mathrm in\beta}$ is real-analytic in $V$.
\end{proof}

The following well-known result provides an explicit form of a quasi-periodic Green function, provided that the frequency $k$ is away from the set of Wood anomaly frequencies $\Lambda:=\{ k\in\mathbb R : k^2-k_s^2 = 0 \text{ for some }s\in\mathbb Z \}$. This is obtained by a standard Rayleigh expansion argument; see \cite{nedelec1991integral} for an analogous construction in the three-dimensional setting.

\begin{proposition}\label{thm:morishita}
    Let $k\in\mathbb R\setminus\Lambda$. 
    Then the series
    \begin{align}\label{eq:poisson}
        G_k(x) = \frac{\mathrm i}{2L}\sum_{s=-\infty}^\infty \frac{1}{\sqrt{k^2-k_s^2}} \mathrm e^{\mathrm ik_sx_1 + \mathrm i\sqrt{k^2-k_s^2}|x_2|}
    \end{align}
    converges in $L^p_\mathrm{loc}(\mathbb R^2)$ for all $1\leq p<\infty$, and the limit $G_k$ is a quasi-periodic Green function. 
\end{proposition}

It is easy to see that the quasi-periodic Green function \cref{eq:poisson} satisfies both the upward and downward radiation conditions in $\mathbb R^2\setminus\Theta$. The major challenge is that this quasi-periodic Green function cannot be extended continuously to $k\in\Lambda$ even in $\mathcal D^\prime(\mathbb R^2)$, i.e., $(G_k)$ converges neither in $L_\mathrm{loc}^p(\mathbb R^2)$ nor in $\mathcal D^\prime(\mathbb R^2)$ as $k$ approaches a Wood anomaly frequency through $\mathbb R\setminus\Lambda$.

An important observation is that a quasi-periodic Green function can still be constructed at $k\in\Lambda$. 

\begin{proposition}\label{prop:green-modified}
    Let $k\in\mathbb R$. Then the series
    \begin{align}\label{eq:green-all}
        G_k(x) &= \frac{\mathrm i}{2L}\sum_{s\in\mathbb Z\setminus\sigma_\ast(k)} \frac{\mathrm e^{\mathrm ik_sx_1+\mathrm i\sqrt{k^2-k_s^2}|x_2|}}{\sqrt{k^2-k_s^2}} - \frac{1}{2L}\sum_{s\in\sigma_\ast(k)} |x_2|\mathrm e^{\mathrm ik_sx_1}
    \end{align}
    converges in $L^p_\mathrm{loc}(\mathbb R^2)$ for all $1\leq p<\infty$, and the limit $G_k$ is a quasi-periodic Green function.
\end{proposition}
\begin{proof}
    The convergence of the series is a complete analog of \cref{thm:morishita} since $\sigma_\ast(k)$ is finite. Thus it suffices to show that the limit $G_k$ satisfies $(-\varDelta-k^2)G_k = F_\#$.

    Consider the partial sum
    \begin{align*}
        G^N_k(x) = \frac{\mathrm i}{2L}\sum_{s\in\{-N,\ldots,N\}\setminus\sigma_\ast(k)} \frac{\mathrm e^{\mathrm ik_sx_1+\mathrm i\sqrt{k^2-k_s^2}|x_2|}}{\sqrt{k^2-k_s^2}} - \frac{1}{2L}\sum_{s\in\sigma_\ast(k)} |x_2|\mathrm e^{\mathrm ik_sx_1}
    \end{align*}
    for sufficiently large $N\in\mathbb N$ with $\sigma_\ast(k)\subset \{-N,\ldots,N\}$. A straightforward calculation shows that
    \begin{align*}
        (-\varDelta-k^2)G^N_k(x)  = \frac{1}{L}\tilde\delta_2 \sum_{s=-N}^N \mathrm e^{\mathrm ik_sx_1},
    \end{align*}
    where $\tilde \delta_2\in\mathcal D^\prime(\mathbb R^2)$ is defined by
    \begin{align*}
        \langle\tilde\delta_2,\phi\rangle = \int_{-\infty}^\infty \phi(x_1,0)\mathrm dx_1 \quad\text{for all }\phi\in C_c^\infty(\mathbb R^2).
    \end{align*}
    Using
    \begin{align*}
        \frac{1}{L}\tilde\delta_2 \sum_{s=-N}^N \mathrm e^{\mathrm ik_sx_1} \ \longrightarrow \ F_\# \quad\text{in }\mathcal D^\prime(\mathbb R^2)
        \quad\text{and}\quad
        G^N_k \ \longrightarrow \ G_k \quad\text{in } L^p_\mathrm{loc}(\mathbb R^2),
    \end{align*}
    we obtain $(-\varDelta-k^2) G_k = F_\#$. The quasi-periodicity is immediate.
\end{proof}

For $k\notin\Lambda$, the Green function in \cref{prop:green-modified} coincides with the radiating quasi-periodic Green function in \cref{thm:morishita}. For $k\in\Lambda$, it is a non-radiating quasi-periodic Green function.

\subsection{Boundary integral formulation}
Our first result establishes a boundary integral formulation of the scattering problem $(P_k)$ for arbitrary $k\in\mathbb R$, once a finite set $\sigma\subset\mathbb Z$ containing the Wood-anomaly indices has been chosen.

Let $\sigma$ be a finite subset of $\mathbb Z$ and define $\Lambda_\sigma:= \{k\in\Lambda:\sigma_\ast(k)\subset\sigma\}$.  Using the Green function $G_k$, defined in \cref{prop:green-modified}, we define
\begin{align*}
    \hat G_k(x) :=
            G_k(x) - \frac{\mathrm i}{2L}\sum_{s\in\sigma\setminus\sigma_\ast(k)} \frac{\mathrm e^{\mathrm ik_sx_1}}{\sqrt{k^2-k_s^2}}  
\end{align*}
for $x\in\mathbb R^2\setminus\Theta$ and $k\in(\mathbb R\setminus\Lambda)\cup \Lambda_\sigma$. 
Although $\hat G_k$ is not itself a Green function when $\sigma\setminus\sigma_\ast(k)\neq\emptyset$, the finite-dimensional correction introduced below cancels the non-Helmholtz components generated by the subtraction.

For $\varphi\in C_\beta(\Gamma_\mathrm{per})$, $k\in (\mathbb R\setminus\Lambda)\cup \Lambda_\sigma$, and $x\in\Gamma_\mathrm{per}$, we define the integrals
\begin{align*}
    (S_k\varphi)(x) := 2\int_\Gamma \hat G_k(x-y)\varphi(y)\mathrm ds(y),\quad
    (D_k\varphi)(x) := 2\int_\Gamma \frac{\partial \hat G_k}{\partial \nu(y)}(x-y)\varphi(y)\mathrm ds(y).
\end{align*}
In view of \cref{lemma:green-regularity,prop:green-modified}, the kernels of these boundary integrals are at most weakly singular. Indeed, the logarithmic singularity of $G_k$ gives the usual weak singularity of $S_k$, while the $C^{1,1}$ regularity of $\Gamma_\mathrm{per}$ makes the normal derivative kernel of $D_k$ weakly singular. Hence the standard compactness results for layer-potential operators apply. Thus the integrals define compact linear operators $S_k,D_k:C_\beta(\Gamma_\mathrm{per})\to C_\beta(\Gamma_\mathrm{per})$ for each $k\in (\mathbb R\setminus\Lambda)\cup \Lambda_\sigma$ (see \cref{prop:cts-lipschitz}).

Define $\mathcal C_\sigma := \{ (c_s)_{s\in\sigma}: c_s\in\mathbb C \}$ with $\|c\|_{\mathcal C_\sigma}:=\max_{s\in\sigma} |c_s|$ and compact operators $T:C_\beta(\Gamma_\mathrm{per})\to \mathcal C_\sigma$ and $R:\mathcal C_\sigma\to C_\beta(\Gamma_\mathrm{per})$ by
\begin{align*}
    (T\varphi)_s = \frac{\mathrm i}{L}\int_\Gamma \left( (\partial_\nu - \mathrm i\eta) \mathrm e^{-\mathrm ik_sy_1} \right) 
    \varphi(y) \mathrm ds(y),\quad 
    (R c)(x) = \sum_{s\in \sigma} c_s\mathrm e^{\mathrm ik_sx_1}.
\end{align*}
Combining these operators, we consider the linear operator $A_k:C_\beta(\Gamma_\mathrm{per})\times \mathcal C_\sigma\to C_\beta(\Gamma_\mathrm{per})\times \mathcal C_\sigma$ defined by
\begin{align*}
    A_k := \begin{bmatrix}
            D_k - \mathrm i\eta S_k & R
        \\
            T & -\mathrm{diag} (1+\sqrt{k^2-k_s^2} )_{s\in\sigma}
        \end{bmatrix}
        ,
\end{align*}
where $\eta>0$ is a constant.

\begin{theorem}\label{thm:main-1}
    Let $k\in\mathbb R$, let $\eta>0$, and assume that $(P_k)$ admits only the trivial solution when $g=0$.
    Then 
    \begin{enumerate}\renewcommand{\theenumi}{\roman{enumi}}
        \item the problem $(P_k)$ is uniquely solvable for all $g\in C_\beta(\Gamma_\mathrm{per})$, 
        \item for each finite subset $\sigma\subset\mathbb Z$ with $k\in (\mathbb R\setminus\Lambda)\cup \Lambda_\sigma$, the boundary integral equation 
        \begin{align*}
            (Q_k) \ \text{Find }(\varphi,c)\in C_\beta(\Gamma_\mathrm{per})\times \mathcal C_\sigma \text{ such that }
            (I+A_k)
            \begin{pmatrix}
                \varphi \\ c
            \end{pmatrix}
            =
            \begin{pmatrix}
                2g \\ 0
            \end{pmatrix}
        \end{align*}
        is uniquely solvable for all $g\in C_\beta(\Gamma_\mathrm{per})$, and its solution $(\varphi,c)$ gives the unique solution $u_k$ of $(P_k)$ by
        \begin{align*}
            u_k(x) = \int_\Gamma \left(\frac{\partial\hat G_k}{\partial\nu(y)}(x-y) -\mathrm i\eta \hat G_k(x-y) \right) \varphi(y) \mathrm ds(y) + \frac{1}{2}\sum_{s\in\sigma} c_s\mathrm e^{\mathrm ik_sx_1}
        \end{align*}
        for $x\in\Omega_+$.
    \end{enumerate}
    In addition, let $J\subset\mathbb R$ be an interval and assume that the homogeneous problem $(P_k)$ has only the trivial solution for all $k\in J$. Then, for every continuous family $\{g_k\}_{k\in J}\subset C_\beta(\Gamma_\mathrm{per})$, the corresponding solutions $u_k$ of $(P_k)$ depend continuously on $k$ in $C_{\rm loc}(\overline\Omega_+)$, i.e., for each $k\in J$ and compact subset $V\subset\overline \Omega_+$, it holds that
    \begin{align*}
        \sup_{x\in V} |u_k(x) - u_{\kappa}(x)| \longrightarrow 0\quad
    \text{as }\kappa\to k.
    \end{align*}
\end{theorem}

\begin{remark}
    By \cref{thm:kirsch}, the overhang-free condition $\nu_2>0$ is sufficient for the hypothesis of \cref{thm:main-1}.
\end{remark}

\subsection{Nystr\"om discretization and error estimate}

\cref{thm:main-1} motivates us to construct numerical approximations of a unique solution to the boundary integral equation $(Q_k)$. We shall establish an error estimate that is locally uniform in $k$ and bound on the condition number of the discrete system. The following definitions provide a self-contained description of the fully discrete method; in particular, the scheme can be implemented directly from the formulas below. Its derivation and convergence analysis are given in Section 5.

Let $\chi\in C_c^\infty(\mathbb R)$ be a window function such that $0\le\chi\le1$, $\chi(s) = 1$ for all $s\in [-1,1]$ and $\chi(s) = 0$ for all $|s|\ge 2$.
Let $\zeta>0$ and define
\begin{align*}
    H_{k,m}(t,\tau) 
    &:= -2|\gamma'(\tau)| \mathrm e^{-\mathrm i\beta(t-\tau)/(2\pi)} (\nu(\gamma(\tau))\cdot \nabla + \mathrm i\eta) \left.W_{k,m}(x)\right|_{x=\gamma(t)-\gamma(\tau)}
    \\
    &\hspace{-10pt}- \frac{|\gamma^\prime(\tau)|}{2\pi}\sum_{\ell\in\mathbb Z}
    \chi \left( \frac{t-\tau-2\pi\ell}{2\pi m} \right)
    \mathrm e^{-\mathrm i\beta(t-\tau-2\pi\ell)/(2\pi)}
    \biggl\{
        - \mathrm i\eta E_1\!\left(\zeta^2(t-\tau-2\pi\ell)^2\right)
    \\
        &+
        (\nu(\gamma(\tau))\cdot \nabla_x + \mathrm i\eta)\biggl[
            E_1\!\left(\zeta^2|x|^2\right)J_0(k|x|)
            \\
            &+
            \mathrm e^{-\zeta^2|x|^2} \sum_{j=1}^{m}\frac{1}{(j!)^2}\left(\frac{k}{2\zeta}\right)^{2j}\sum_{q=0}^{j-1}(j-1-q)!(-\zeta^2|x|^2)^q
        \biggr]_{x=\gamma(t)-\gamma(\tau) - \ell Le_1}
    \biggr\},
    \end{align*}
    where
    \begin{align*}
    W_{k,m}(x) &:= 
    \frac{\mathrm i}{4L}
    \sum_{\substack{|s|\le m\\ s\notin\sigma}}
    \frac{\mathrm e^{\mathrm ik_sx_1}}{\sqrt{k^2-k_s^2}}
    \Biggl[
        \mathrm e^{-\mathrm i\sqrt{k^2-k_s^2} x_2}
    \operatorname{erfc}\!\left(
        \zeta x_2-\frac{\mathrm i\sqrt{k^2-k_s^2}}{2\zeta}
    \right)
    \\
    &+
    \mathrm e^{\mathrm i\sqrt{k^2-k_s^2} x_2}
    \operatorname{erfc}\!\left(
        -\zeta x_2-\frac{\mathrm i\sqrt{k^2-k_s^2}}{2\zeta}
    \right)
    \Biggr]
    + \sum_{s\in\sigma}
    \mathrm e^{\mathrm ik_sx_1}
    \Psi(\sqrt{k^2-k_s^2},x_2)
    \end{align*}
    with continuous function
    \begin{align*}
    \Psi(\alpha,x_2)
    &:=
    \begin{cases}
    \displaystyle
    \frac{\mathrm i}{4L}
    \frac{\mathrm e^{-\mathrm i\alpha x_2}
    \operatorname{erfc}\!\left(
        \zeta x_2-\frac{\mathrm i\alpha}{2\zeta}
    \right)
    +
    \mathrm e^{\mathrm i\alpha x_2}
    \operatorname{erfc}\!\left(
        -\zeta x_2-\frac{\mathrm i\alpha}{2\zeta}
    \right)-2}{\alpha}
    & \alpha\ne0,
    \\[2ex]
    \displaystyle
    -\frac{1}{2L}
    \left(
        x_2\operatorname{erf}(\zeta x_2)
        +
        \frac{\mathrm e^{-\zeta^2x_2^2}}{\sqrt\pi\,\zeta}
    \right)
    & \alpha=0
    \end{cases}
\end{align*}
for each $t,\tau\in\mathbb R$, $k\in (\mathbb R\setminus\Lambda)\cup \Lambda_\sigma$, and $m\in\mathbb N$. For each $j\in\mathbb N$, the function $E_j$ is defined by
\begin{align*}
    E_j(z) := \int_1^\infty \mathrm t^{-j} e^{-zt} \mathrm dt \quad\text{for all }z\in\mathbb C \text{ with }\mathrm{Re}[z]>0.
\end{align*}
Note that the sum over $\ell$ is finite for each $t,\tau\in\mathbb R$ due to the compact support of $\chi$. Moreover, if $\gamma\in C^2(\mathbb R)$, then the singularities in the above expression cancel and $H_{k,m}$ extends continuously to $\mathbb T\times\mathbb T$, where $ \mathbb T:= \mathbb R/(2\pi\mathbb Z)$.

Using the function $H_{k,m}$, we consider the following $(2n+\#\sigma)$-by-$(2n+\#\sigma)$ linear system: $(Q_k^{(n,m)})$ Find $(\widetilde{\bm{\varphi}}^{(n,m)},c^{(n,m)})\in \mathbb C^{2n}\times\mathcal C_\sigma$ such that 
\begin{align*}
    \underbrace{
    \begin{bmatrix}
        B^{(n,m)}_{11} & B^{(n)}_{12}
        \\
        B^{(n)}_{21} &  B_{22}
    \end{bmatrix}
    }_{=:B^{(n,m)}_k}
    \begin{pmatrix}
        \widetilde{\bm{\varphi}}^{(n,m)} \\ c^{(n,m)}
    \end{pmatrix}
    =
    \begin{pmatrix}
        2\widetilde{\bm{g}}^{(n)} \\ 0
    \end{pmatrix}
    ,
\end{align*}
where the matrices $B^{(n,m)}_{11}\in \mathbb C^{2n\times 2n}$, $B^{(n)}_{12}\in \mathbb C^{2n\times \#\sigma}$, $B^{(n)}_{21}\in \mathbb C^{\#\sigma\times 2n}$, $B_{22}\in \mathbb C^{\#\sigma\times \#\sigma}$ and the vector $\widetilde{\bm{g}}^{(n)}\in \mathbb C^{2n}$ are defined as follows:
\begin{align*}
    (B^{(n,m)}_{11})_{ij} &:= \delta_{ij} - \frac{\mathrm i\eta}{2\pi} R^{(n)}_j(t^{(n)}_i)|\gamma^\prime(t^{(n)}_j)| + \frac{\pi}{n}H_{k,m}(t^{(n)}_i,t^{(n)}_j),
    \\
    (B^{(n)}_{12})_{is} &:= \mathrm e^{-\mathrm i\beta t^{(n)}_i/(2\pi) + \mathrm ik_s\gamma_1(t^{(n)}_i)},
    \\
    (B^{(n)}_{21})_{sj} &:= \frac{\mathrm i\pi}{nL}
    \mathrm e^{\mathrm i\beta t^{(n)}_j/(2\pi)}
        |\gamma^\prime(t^{(n)}_j)|
        \left(
            \nu(\gamma(t^{(n)}_j))\cdot
            \nabla_y
            -
            \mathrm i\eta
        \right)
        \mathrm e^{-\mathrm ik_sy_1}\big|_{y=\gamma(t^{(n)}_j)},
    \\
    (B_{22})_{ls} &:= -\sqrt{k^2-k_s^2} \delta_{sl},
    \\
    \widetilde g^{(n)}_j &:= \mathrm e^{-\mathrm i\beta t^{(n)}_j/(2\pi)} g(\gamma(t^{(n)}_j))
\end{align*}
for $i,j=0,\ldots,2n-1$ and $s,l\in\sigma$ with
\begin{align}
\label{eq:def-t}
    t^{(n)}_j := j\pi/n \quad\text{for }j=0,1,\ldots,2n-1.
\end{align}
For each $j=0,\ldots,2n-1$, the function $R^{(n)}_j:\mathbb T\to\mathbb C$ is given by
\begin{align}
\label{eq:def-R}
    R_j^{(n)}(t)
    &=
    \frac{\pi}{n}\sum_{p=-n}^{n-1}
    \widehat{\mathcal E}_p\,
    \mathrm e^{\mathrm i p(t-t^{(n)}_j)},
    \\ \notag
    \widehat{\mathcal E}_p &= 
    \begin{cases}
            \dfrac{1}{|p+\beta/(2\pi)|}
            \operatorname{erf}\!\left(\dfrac{|p+\beta/(2\pi)|}{2\zeta}\right)
            & \text{if }p+\beta/(2\pi)\neq 0,\\[2ex]
            \dfrac{1}{\sqrt{\pi}\,\zeta}
            & \text{otherwise}.
        \end{cases}
\end{align}

To state our second theorem, we impose stronger regularity assumptions on the data $g$ and the boundary $\Gamma_\mathrm{per}$ to obtain a convergence rate for the proposed scheme. Define $C^{0,1}_\beta(\Gamma_\mathrm{per}) := \{ \varphi\in C_\beta(\Gamma_\mathrm{per}) : \ \operatorname{Lip}(\varphi\circ\gamma) < +\infty \}$. We equip this space with the norm $\| \varphi \|_{C^{0,1}_\beta(\Gamma_\mathrm{per})} := \| \varphi \|_{C_\beta(\Gamma_\mathrm{per})} + \operatorname{Lip}(\varphi\circ\gamma)$, where $\operatorname{Lip}$ denotes the Lipschitz constant.

\begin{theorem}\label{thm:main-2}
    Let $K$ be a compact subset of $\mathbb R$ and assume that the homogeneous problem $(P_k)$ admits only the trivial solution for all $k\in K$. Let $\eta>0$ and let $\sigma\subset\mathbb Z$ be a finite subset such that $K\subset (\mathbb R\setminus\Lambda)\cup \Lambda_\sigma$. 
    Let $\{g_k\}_{k\in K}$ be such that $k\mapsto g_k$ is continuous from $K$ into $C_\beta^{0,1}(\Gamma_\mathrm{per})$. Suppose that $\gamma\in C^{2,\alpha}(\mathbb R)$ for some $\alpha\in (0,1]$.
    Then the following assertions hold:
    \begin{enumerate}\renewcommand{\theenumi}{\roman{enumi}}
        \item there exist $N\in\mathbb N$ and $M\in\mathbb N$ such that $(Q_{k}^{(n,m)})$ with $g=g_k$ is uniquely solvable for all $k\in K$, $n\ge N$, and $m\ge M$,
        \item there exist $C>0$ and $\xi>0$ such that
        \begin{align*}
            \sup_{k\in K} \left( \| \varphi_k - \varphi^{(n,m)}_k\|_{C_\beta(\Gamma_\mathrm{per})} + \| c_k - c^{(n,m)}_k\|\right) \leq C \left(n^{-\alpha}+\mathrm e^{-\xi m}\right)
        \end{align*}
        for all $n\ge N$ and $m\ge M$,
        \item the $\ell^\infty$ condition number of the matrix $B^{(n,m)}_k$ is uniformly bounded with respect to $k\in K$, $n\ge N$, and $m\ge M$,
    \end{enumerate}
    where $(\varphi_k,c_k)\in C_\beta(\Gamma_\mathrm{per})\times \mathcal C_\sigma$ is the unique solution to $(Q_k)$ with $g=g_k$, and $(\varphi^{(n,m)}_k,c^{(n,m)}_k)\in C_\beta(\Gamma_\mathrm{per})\times \mathcal C_\sigma$ are obtained by the unique solution $(\widetilde{\bm{\varphi}}^{(n,m)}, c^{(n,m)}_k)\in \mathbb C^{2n}\times\mathcal C_\sigma$ to $(Q^{(n,m)}_k)$ with $g=g_k$ as follows:
    \begin{align*}
        \varphi^{(n,m)}_k(\gamma(t)) &= \mathrm e^{\mathrm i\beta t/(2\pi)}\sum_{j=0}^{2n-1}\left( R^{(n)}_j(t)\frac{\mathrm i\eta}{2\pi}|\gamma^\prime(t^{(n)}_j)| - \frac{\pi}{n}H_{k,m}(t,t^{(n)}_j)
        \right)
        (\widetilde{\bm{\varphi}}^{(n,m)}_k)_j 
        \\
        &\qquad - \sum_{s\in\sigma}  (c^{(n,m)}_k)_s\mathrm e^{\mathrm ik_s(\gamma(t))_1} + 2g_k(\gamma(t)).
    \end{align*}
\end{theorem}

\section{Quasi-periodic Green function}
We begin with the analysis of the quasi-periodic Green function $G_k$, defined in \cref{prop:green-modified}. The following standard convergence result, which dates back to Bruno and Reitich \cite{bruno1992solution}, will be used repeatedly.

\begin{lemma}\label{lemma:Gk-unif-conv}
    Let $\sigma\subset\mathbb Z$ be a finite subset. Then the series
    \begin{align*}
        \frac{\mathrm i}{2L}\sum_{s\in\mathbb Z\setminus\sigma} \frac{\mathrm e^{\mathrm ik_sx_1+\mathrm i\sqrt{k^2-k_s^2}|x_2|}}{\sqrt{k^2-k_s^2}}
    \end{align*}
    converges uniformly for $(x,k)$ in compact subsets of $(\mathbb R^2\setminus\Theta) \times ( (\mathbb R\setminus\Lambda)\cup \Lambda_{\sigma})$.
\end{lemma}
Note that absolute convergence fails on compact sets that intersect the line $x_2=0$.

\subsection{Continuity with respect to $k$}

In light of \cref{thm:main-1}, we are particularly interested in the continuity with respect to $k$ at and around Wood anomaly frequencies. 
Although each $G_k$ has logarithmic singularities at the lattice points $\Theta$, these spatial singularities are independent of $k$ and are cancelled in differences of Green functions. The additional difficulty near a Wood anomaly is the singular Rayleigh coefficient of the grazing mode. The following proposition shows that, after subtracting this explicit singular coefficient, the Green function depends continuously on $k$ in $C^{1,\alpha}_{\rm loc}(\mathbb R^2)$.

\begin{proposition}\label{prop:Gk-cts}
    Let $k\in\mathbb R$. Then 
    \begin{align*}
        G_{\kappa} - \frac{\mathrm i}{2L} \sum_{s\in\sigma_\ast(k)} \frac{\mathrm e^{\mathrm ik_sx_1}}{\sqrt{\kappa^2-k_s^2}} - G_k \longrightarrow 0 \quad\text{in }C^{1,\alpha}_\mathrm{loc}(\mathbb R^2) \quad\text{as }\kappa\to k \text{ with }\kappa\in\mathbb R\setminus\Lambda
    \end{align*}
    for all $\alpha\in (0,1)$. Moreover, for each $N\in\mathbb N\cup \{0\}$ and $\alpha\in (0,1)$, the function $H_k^N(x) := G_k(x) - \sum_{n=-N}^N \Phi^{(n)}_k(x) \mathrm e^{\mathrm in\beta}$ satisfies 
    \begin{align*}
        \nabla \left(H^N_{\kappa} - \frac{\mathrm i}{2L} \sum_{s\in\sigma_\ast(k)} \frac{\mathrm e^{\mathrm ik_sx_1}}{\sqrt{\kappa^2-k_s^2}} \right) - \nabla H^N_k \longrightarrow 0 \quad\text{in }C^{1,\alpha}_\mathrm{loc}(V_N) \quad\text{as }\kappa\to k \text{ with }\kappa\in\mathbb R\setminus\Lambda,
    \end{align*}
    where $V_N:=\mathbb (R^2\setminus\Theta)\cup \{-NLe_1,\ldots,NLe_1\}$.
\end{proposition}
\begin{proof}
    Let $(\kappa_n)\subset\mathbb R\setminus\Lambda$ be a sequence such that $\kappa_n\to k$. Then
    \begin{align*}
        v_n:= G_{\kappa_n} - \frac{\mathrm i}{2L} \sum_{s\in\sigma_\ast(k)} \frac{\mathrm e^{\mathrm ik_sx_1}}{\sqrt{\kappa_n^2-k_s^2}} - G_k
    \end{align*}
    satisfies the following equation for each $n\in\mathbb N$:
    \begin{align}\label{prop:Gk-cts:1}
            (-\varDelta - \kappa_n^2)v_n = (\kappa_n^2-k^2)G_k + \frac{\mathrm i}{2L}\sum_{s\in\sigma_\ast(k)}\sqrt{\kappa_n^2-k_s^2} \mathrm e^{\mathrm ik_sx_1} \quad\text{in }\mathbb R^2. 
    \end{align}

    Choose $R>0$ small enough so that  $\sup_{n\in\mathbb N}\kappa_n^2$ is less than the smallest eigenvalue of the negative Dirichlet Laplacian $-\varDelta$ in $B(0;R)$. Let $\rho\in(0,R)$ and let $V\Subset\mathbb R^2$. Since $\overline V$ is compact and $\Theta$ is discrete, there exist finitely many points $y_1,\ldots,y_J\in\mathbb R^2$ such that 
    \begin{align*}
            \overline V\subset \bigcup_{j=1}^J B(y_j;\rho)\quad \text{and}\quad
            \Theta\cap \partial B(y_j;R) = \emptyset. 
    \end{align*}
    Let $j\in\{1,\ldots,J\}$ be fixed. In view of the finite covering, it suffices to show that $v_n\to 0$ in $C^{1,\alpha}(B[y_j;\rho])$. Since $G_k\in L^p(B(y_j;R))$ by \cref{prop:green-modified}, the right-hand side of \cref{prop:Gk-cts:1} tends to zero in $L^p(B(y_j;R))$ as $n\to\infty$ for all $p\in [1,\infty)$. 
    
    We assert that
    \begin{align}\label{prop:Gk-cts:2}
        v_n \longrightarrow 0 \quad\text{in }H^{1/2}(\partial B(y_j;R)).
    \end{align}
    To see this, it suffices to show that $v_n\to 0$ in $H_\mathrm{loc}^1(\mathbb R^2\setminus\Theta)$ since $\Theta\cap \partial B(y_j;R) = \emptyset$. We have
    \begin{align*}
        v_n &= \frac{\mathrm i}{2L}\sum_{s\in\mathbb Z\setminus\sigma_\ast(k)} \left( 
            \frac{\mathrm e^{\mathrm ik_sx_1+\mathrm i\sqrt{\kappa_n^2-k_s^2}|x_2|}}{\sqrt{\kappa_n^2-k_s^2}} - \frac{\mathrm e^{\mathrm ik_sx_1+\mathrm i\sqrt{k^2-k_s^2}|x_2|}}{\sqrt{k^2-k_s^2}}
        \right)
    \\
        &\qquad + \frac{\mathrm i}{2L}\sum_{s\in\sigma_\ast(k)} \left(
            \frac{\mathrm e^{\mathrm ik_sx_1+\mathrm i\sqrt{\kappa_n^2-k_s^2}|x_2|}}{\sqrt{\kappa_n^2-k_s^2}} - \frac{\mathrm e^{\mathrm i k_sx_1}}{\sqrt{\kappa_n^2-k_s^2}} - \mathrm i|x_2|\mathrm e^{\mathrm ik_sx_1}
        \right), 
    \end{align*}
    where the first and second series respectively vanish as $n\to\infty$ uniformly in compact subsets of $\mathbb R^2\setminus\Theta$ due to \cref{lemma:Gk-unif-conv} and the fact that $(\mathrm e^{\mathrm iz t}-1)/z - \mathrm{i}t \to 0$ as $z\to 0$ uniformly for $t$ in compact subsets of $\mathbb R$. In particular, $v_n\to 0$ in $L^2_\mathrm{loc}(\mathbb R^2\setminus\Theta)$. In addition, we have
    \begin{align}\label{prop:Gk-cts:3}
         (\kappa_n^2-k^2)G_k + \frac{\mathrm i}{2L}\sum_{s\in\sigma_\ast(k)}\sqrt{\kappa_n^2-k_s^2} \mathrm e^{\mathrm ik_sx_1} \longrightarrow 0 \quad\text{in }L^p_\mathrm{loc}(\mathbb R^2) \quad\text{for all }p\in [1,\infty).
    \end{align}
    We apply an a priori $H^1$ estimate to \cref{prop:Gk-cts:1} and obtain $v_n\to 0$ in $H_\mathrm{loc}^1(\mathbb R^2\setminus\Theta)$, which implies \cref{prop:Gk-cts:2}. 

    By the assumption on $R$, we have an $n$-uniform coercivity for the Dirichlet problem; thus the convergence \cref{prop:Gk-cts:2,prop:Gk-cts:3} imply $v_n\to 0$ in $H^1(B(y_j;R))$ and hence in $L^p(B(y_j;R))$ for all $p\in [1,\infty)$. Therefore a standard interior $W^{2,p}$ estimate shows that $v_n\to 0$ in $W^{2,p}(B(y_j;\rho))$ for all $j=1,\ldots,J$ and arbitrarily large $p<\infty$. Therefore, the Morrey--Sobolev inequality and the open ball covering yield $v_n\to 0$ in $C^{1,\alpha}(\overline V)$ for any $\alpha\in(0,1)$.

    We next show the continuity of $\nabla H^N_k$ in $C^{1,\alpha}_\mathrm{loc}(V_N)$ with finite-term correction. Define
    \begin{align*}
        w_\kappa(x) := H^N_{\kappa}(x) - \frac{\mathrm i}{2L} \sum_{s\in\sigma_\ast(k)} \frac{\mathrm e^{\mathrm ik_sx_1}}{\sqrt{\kappa^2-k_s^2}} - H^N_{k}(x).
    \end{align*}
    Then it solves
    \begin{align*}
        (-\varDelta-\kappa^2)w_\kappa = (\kappa^2-k^2)H^N_k + \frac{\mathrm i}{2L}\sum_{s\in\sigma_\ast(k)}\sqrt{\kappa^2-k_s^2}\mathrm e^{\mathrm ik_sx_1} =:h_\kappa \quad\text{in }\mathcal D^\prime(V_N).
    \end{align*}
    In view of $\nabla (\Phi^{(n)}_\kappa-\Phi^{(n)}_k) \to 0$ in $C^{0,\alpha}_\mathrm{loc}(\mathbb R^2)$, the previous assertion yields that $\nabla w_\kappa \to 0$ in $C^{0,\alpha}_\mathrm{loc}(\mathbb R^2)$ as $\kappa\to k$, whereas $\nabla h_\kappa\to 0$ in $L^{p}_\mathrm{loc}(V_N)$. Thus $\nabla w_\kappa \to 0$ in $W^{2,p}_\mathrm{loc}(V_N)$, i.e., in $C^{1,\alpha}_\mathrm{loc}(V_N)$.
\end{proof}

The continuity of the boundary integral operators with respect to $k$ is established from this result, summarized as follows.

\begin{proposition}\label{prop:cts-lipschitz}
    Let $\sigma\subset\mathbb Z$ be a finite subset.
    Then the linear operators $S_k$ and $D_k$ from $C_\beta(\Gamma_\mathrm{per})$ into itself are compact for each $k\in(\mathbb R\setminus\Lambda)\cup\Lambda_\sigma$, and $k\mapsto S_k$ and $k\mapsto D_k$ are continuous from $(\mathbb R\setminus\Lambda)\cup\Lambda_\sigma$ into $\mathcal L(C_\beta(\Gamma_\mathrm{per}),C_\beta(\Gamma_\mathrm{per}))$.
    Moreover, if $\gamma\in C^{2,\alpha}(\mathbb R)$ for some $\alpha\in (0,1]$, then $S_k$ and $D_k$ from $C^{0,1}_\beta(\Gamma_\mathrm{per})$ into itself are compact for each $k\in(\mathbb R\setminus\Lambda)\cup\Lambda_\sigma$, and $k\mapsto S_k$ and $k\mapsto D_k$ are continuous from $(\mathbb R\setminus\Lambda)\cup\Lambda_\sigma$ into $\mathcal L(C^{0,1}_\beta(\Gamma_\mathrm{per}),C^{0,1}_\beta(\Gamma_\mathrm{per}))$.
\end{proposition}
\begin{proof}
    The mapping property of $S_k$ and $D_k$ on $C_\beta(\Gamma_\mathrm{per})$ and their continuity with respect to $k$ are immediate from \cref{lemma:green-regularity,prop:Gk-cts}.

    Assume $\gamma\in C^{2,\alpha}(\mathbb R)$ and set $D=(\mathbb R\setminus\Lambda)\cup\Lambda_\sigma$ and
    $\Omega=\{ x-y:x,y\in\Gamma\}$.  Choose $N$ so large that
    $\Omega\Subset V_N:=\mathbb R^2\setminus
    \Theta_{\mathbb Z\setminus\{-N,\ldots,N\}}$.  We use the elementary kernel
    criterion that, if $Q\in C^1(J;L^1(\Gamma))$, $J=[0,2\pi]$, then
    \[
     (T_Q\varphi)(t):=\int_\Gamma Q(t,y)\varphi(y)\,ds(y),\qquad
     T_Q:L^\infty(\Gamma)\longrightarrow C^{0,1}(J)
    \]
    is compact. 
    We first split $S_k = S_0^N + (S_k - S_0^N)$ and $D_k = D_0^N + (D_k - D_0^N)$, where
    \begin{align*}
        (S_0^N\varphi)(x) &:= 2\sum_{n=-N}^N \mathrm e^{\mathrm in\beta} \int_\Gamma \Phi^{(n)}_0(x-y)\varphi(y)\mathrm ds(y),
        \\
        (D_0^N\varphi)(x) &:= 2\sum_{n=-N}^N \mathrm e^{\mathrm in\beta} \int_\Gamma \frac{\partial \Phi^{(n)}_0}{\partial\nu(y)}(x-y)\varphi(y)\mathrm ds(y).
    \end{align*}
    For each $n=-N,\ldots,N$, there exists a $C^1$ function $q_n$ such that $|\gamma(t)-\gamma(\tau)-nLe_1|=|t-\tau-2\pi n|q_n(t,\tau)$ and $q_n\ge c>0$ for some $c>0$, from which we confirm that $S^{N}_0$ is compact from $C^{0,1}_\beta(\Gamma)$ into $C^{0,1}(\Gamma)$. For the Laplace double-layer kernel
    \begin{align*}
     \mathcal K(t,\tau):=\frac1{\pi}
     \frac{(\gamma(t)-\gamma(\tau))\cdot\nu(\gamma(\tau))}
          {|\gamma(t)-\gamma(\tau)|^2},
    \end{align*}
    Taylor expansion gives the continuous diagonal value $\mathcal K(t,t)=\frac1{2\pi}\frac{\gamma''(t)\cdot\nu(\gamma(t))}{|\gamma'(t)|^2}$ and, by cancellation of the cubic terms, we have $|\partial_t\mathcal K(t,\tau)|\le C|t-\tau|^{\alpha-1}$. Hence $t\mapsto\mathcal K(t,\cdot)$ belongs to $C^1(J;L^1(J))$, and the criterion above shows that $D_0^N$ is compact into $C^{0,1}(\Gamma)$.
    
    It remains to treat the regular remainder. With
    $R_k:=\Phi_k^{(0)}-\Phi_0^{(0)}$, define
    \begin{align*}
     \mathcal M_k^N(r) := H_k^N(r)
     +\sum_{n=-N}^{N}\mathrm e^{\mathrm i n\beta}R_k(r-nLe_1)
     -\frac{\mathrm i}{2L}\sum_{s\in\sigma\setminus\sigma_*(k)}
     \frac{\mathrm e^{\mathrm i k_sr_1}}{\sqrt{k^2-k_s^2}}
    \end{align*}
    on a neighborhood $U$ of $\Omega$ with $\overline U\Subset V_N$. Let $k_0\in(\mathbb R\setminus\Lambda)\cup\Lambda_\sigma$. The first assertion of \cref{prop:Gk-cts} gives $t\mapsto(\mathcal M_k^N-\mathcal M_{k_0}^N)(\gamma(t)-\cdot)\to0$ in $C^1(J;L^1(\Gamma))$ as $k\to k_0$. For the double-layer kernel, observe that
    \begin{align*}
     \frac{d}{dt}
     \left[
     \nu(y)\cdot\nabla (\mathcal M_k^N-\mathcal M_{k_0}^N) (\gamma(t)-y)
     \right]
     =
     \nu(y)\cdot
     D^2 (\mathcal M_k^N-\mathcal M_{k_0}^N) (\gamma(t)-y)\gamma'(t).
    \end{align*}
    The second assertion of \cref{prop:Gk-cts} controls the Hessian of the
    regular part $H_k^N$, while the estimates $|\nabla R_k(r)|\le C|r|(1+|\log|r||)$ and $|D^2R_k(r)|\le C(1+|\log|r||)$ give
    \begin{align*}
     \sup_{t\in J}\int_\Gamma
     \left(
     |\nabla (\mathcal M_k^N-\mathcal M_{k_0}^N) (\gamma(t)-y)|
     +
     |D^2 (\mathcal M_k^N-\mathcal M_{k_0}^N) (\gamma(t)-y)|
     \right) \mathrm ds(y)\longrightarrow0.
    \end{align*}
    Consequently, $t\mapsto\nu(\cdot)\cdot
     \nabla (\mathcal M_k^N-\mathcal M_{k_0}^N) (\gamma(t)-\cdot)\to0$ in $C^1(J;L^1(\Gamma))$.
    
    Therefore, the associated operators $S_k - S_0^N$ and $D_k - D_0^N$ are compact from $C^{0,1}_\beta(\Gamma)$ into itself, and their families are
    continuous in operator norm.
\end{proof}

\subsection{Ewald's method}\label{ss:ewald}
Our boundary integral formulation is based on the quasi-periodic Green function $G_k$, defined in \cref{eq:green-all}. This series representation is, however, inconvenient for the Nystr\"om method since the logarithmic singularity at the lattice $\Theta$ is not explicit. To remedy this, we derive another formula for $G_k$ using Ewald's method. For details, see \cite{arens2013analysing} and references therein.

Consider the following function:
\begin{align}
    \label{eq:ewald}
    G^\mathrm{E}_k(x) = {}&
    \frac{1}{4\pi}\sum_{n=-\infty}^\infty \mathrm{e}^{\mathrm in\beta} \sum_{j=0}^\infty \frac{1}{j!}\left(\frac{k}{2\zeta}\right)^{2j} E_{j+1}(\zeta^2 ((x_1-nL)^2 + x_2^2))
    \\
    &+ \frac{\mathrm i}{4L} \sum_{s\in\mathbb Z\setminus\sigma_\ast(k)} \frac{\mathrm{e}^{\mathrm ik_s x_1} }{\sqrt{k^2 - k_s^2}}
    \biggl[
    \mathrm{e}^{-\mathrm i \sqrt{k^2 - k_s^2} x_2} \mathrm{erfc}\, \left( \zeta x_2 - \frac{\mathrm i\sqrt{k^2 - k_s^2}}{2\zeta} \right)
    \notag
    \\
    &+
        \mathrm{e}^{+\mathrm i \sqrt{k^2 - k_s^2} x_2} \mathrm{erfc}\, \left(-\zeta x_2 - \frac{\mathrm i\sqrt{k^2 - k_s^2}}{2\zeta} \right)
    \Bigg]
    \notag
    \\
    &- \frac{1}{2L}\sum_{s\in\sigma_\ast(k)} \mathrm e^{\mathrm i k_s x_1}
    \left(
        x_2 \mathrm{erf}(\zeta x_2)
        +
        \frac{\mathrm e^{-\zeta^2 x_2^2}}{\sqrt\pi\,\zeta}
    \right)
\notag
\end{align}

The next proposition shows that $G^\mathrm{E}_k$ coincides with $G_k$ away from the lattice $\Theta$.

\begin{proposition}\label{prop:ident-G-GE}
    Let $k\in\mathbb R$ and let $\zeta>0$. Then $G_k(x) = G^\mathrm{E}_k(x)$ for all $x\in\mathbb R^2\setminus\Theta$.
\end{proposition}
\begin{proof}
    (i) We first show that the asserted equality holds for all $x\in\mathbb R^2$ with $x_2\neq 0$. To this end, we claim that
    \begin{align*}
        \theta\longmapsto f(\theta) := \frac{1}{4\pi}\mathrm{e}^{\mathrm i\theta\beta} \sum_{j=0}^\infty \frac{1}{j!}\left(\frac{k}{2\zeta}\right)^{2j} E_{j+1}(\zeta^2 ((x_1-\theta L)^2 + x_2^2))
    \end{align*}
    is a Schwartz function on $\mathbb R$, i.e., 
    $f\in C^\infty(\mathbb R;\mathbb C)$ and $\sup_{\theta\in\mathbb R} |\theta^n f^{(m)}(\theta)| < +\infty$ for all $n,m\in\mathbb N\cup \{0\}$, 
    for each $x\in\mathbb R^2$ with $x_2\neq 0$. In fact, Tonelli's theorem yields
    \begin{align*}
        &\quad\sum_{j=0}^\infty \frac{1}{j!}\left(\frac{k}{2\zeta}\right)^{2j} E_{j+1}(\zeta^2 ((x_1-\theta L)^2 + x_2^2)) 
        = \int_1^\infty t^{-1}\mathrm e^{-\zeta^2 ((x_1-\theta L)^2 + x_2^2)t+k^2/(4\zeta^2 t)}\,\mathrm dt
        \\
        &\le \mathrm e^{k^2/(4\zeta^2)} \int_1^\infty \mathrm e^{-\zeta^2 ((x_1-\theta L)^2 + x_2^2)t} \mathrm dt < +\infty \quad\text{for all }\theta\in\mathbb R,\ x\in\mathbb R^2\text{ with }x_2\neq 0.
    \end{align*}
    By the dominated convergence theorem, we may differentiate $f$ under the integral sign and obtain the estimate
    \begin{align*}
        |\theta|^n|f^{(m)}(\theta)|
        \le
        C_{m,n}\int_1^\infty t^{N_{m,n}-1}\mathrm e^{-\zeta^2x_2^2 t}\,\mathrm dt < +\infty \quad\text{for all }\theta\in\mathbb R \text{ and }m,n\in\mathbb N\cup \{0\}
    \end{align*}
    where $C_{m,n}>0$ and $N_{m,n}\in\mathbb N$ are constants. Thus, by the Poisson summation formula, it follows that
    \begin{align*}
        &\quad \frac{1}{4\pi}\sum_{n=-\infty}^\infty \mathrm{e}^{\mathrm in\beta} \sum_{j=0}^\infty \frac{1}{j!}\left(\frac{k}{2\zeta}\right)^{2j} E_{j+1}(\zeta^2 ((x_1-nL)^2 + x_2^2))
        \\
        &=
        \sum_{s=-\infty}^\infty
        \frac{\mathrm e^{\mathrm ik_sx_1}}{4\sqrt\pi\,L\zeta}
        \int_1^\infty
        t^{-3/2}
        \exp\!\left(
            -\zeta^2x_2^2 t
            +
            \frac{k^2-k_s^2}{4\zeta^2 t}
        \right)\,\mathrm dt \quad\text{for all }x\in\mathbb R^2\text{ with }x_2\neq 0.
    \end{align*}
    We split the series over $s$ into two parts: $s\in\mathbb Z\setminus\sigma_\ast(k)$ and $s\in\sigma_\ast(k)$. Straightforward calculations show that
    \begin{align*}
        &\sum_{s\in\mathbb Z\setminus\sigma_\ast(k)}
        \frac{\mathrm e^{\mathrm ik_sx_1}}{4\sqrt\pi\,L\zeta}
        \int_1^\infty
        t^{-3/2}
        \exp\!\left(
            -\zeta^2x_2^2 t
            +
            \frac{k^2-k_s^2}{4\zeta^2 t}
        \right)\,\mathrm dt
        \\
        &=
        -
        \frac{\mathrm i}{4L} \sum_{s\in\mathbb Z\setminus\sigma_\ast(k)} \frac{\mathrm{e}^{\mathrm ik_s x_1} }{\sqrt{k^2 - k_s^2}}
        \biggl[
        \mathrm{e}^{-\mathrm i \sqrt{k^2 - k_s^2} x_2} \mathrm{erfc}\, \left( \zeta x_2 - \frac{\mathrm i\sqrt{k^2 - k_s^2}}{2\zeta} \right)
        \\
        &+
        \mathrm{e}^{+\mathrm i \sqrt{k^2 - k_s^2} x_2} \mathrm{erfc}\, \left(-\zeta x_2 - \frac{\mathrm i\sqrt{k^2 - k_s^2}}{2\zeta} \right)
        \biggr] + \frac{\mathrm i}{2L}\sum_{s\in\mathbb Z\setminus\sigma_\ast(k)}
        \frac{\mathrm e^{\mathrm ik_sx_1+\mathrm i\sqrt{k^2-k_s^2} x_2}}{\sqrt{k^2-k_s^2}}
    \end{align*}
    and
    \begin{align*}
        &\sum_{s\in\sigma_\ast(k)}
        \frac{\mathrm e^{\mathrm ik_sx_1}}{4\sqrt\pi\,L\zeta}
        \int_1^\infty
        t^{-3/2}
        \exp\!\left(
            -\zeta^2x_2^2 t
            +
            \frac{k^2-k_s^2}{4\zeta^2 t}
        \right)\,\mathrm dt
        \\
        &= \frac{1}{2L}
        \sum_{s\in\sigma_\ast(k)}
        \mathrm e^{\mathrm ik_sx_1}
        \left(
            \frac{\mathrm e^{-\zeta^2x_2^2}}{\sqrt\pi\,\zeta}
            +
            x_2\operatorname{erf}(\zeta x_2)
            - x_2
        \right)
    \end{align*}
    for all $x\in\mathbb R^2$ with $x_2>0$. Thus $G_k(x) = G^\mathrm{E}_k(x)$ for all $x\in\mathbb R^2$ with $x_2>0$. Analogously, we can show that the equality holds for $x\in\mathbb R^2$ with $x_2<0$.

    (ii) We next assert that $G^\mathrm{E}_k$ is continuous in $\mathbb R^2\setminus\Theta$. For every compact subset $V$ of $\mathbb R^2\setminus\Theta$, there exist constants $c,C>0$ such that
    \begin{align*}
        &\quad\sum_{j=0}^\infty \frac{1}{j!}\left(\frac{k}{2\zeta}\right)^{2j} E_{j+1}(\zeta^2 ((x_1-nL)^2 + x_2^2)) 
        \\
        &\le \mathrm e^{k^2/(4\zeta^2)} \int_1^\infty \mathrm e^{-\zeta^2 ((x_1-n L)^2 + x_2^2)t} \mathrm dt 
        \le C e^{-c n^2}
        \quad\text{for all }n\in\mathbb Z \text{ and } x\in V,
    \end{align*}
    where $d:= \min_{x\in V}\mathrm{dist}(x,\Theta) > 0$. Hence the first series in \cref{eq:ewald} converges absolutely and uniformly for $x$ in compact subsets and thus defines a continuous function on $\mathbb R^2\setminus\Theta$. 
    Analogously, due to the exponential decay of $\operatorname{erfc}(x) \sim \mathrm e^{-x^2}/x$ for large $x>0$, the second series in \cref{eq:ewald} defines a continuous function as well. Therefore $G^\mathrm{E}_k$ is continuous on $\mathbb R^2\setminus\Theta$.

    From (i), (ii), and \cref{lemma:Gk-unif-conv}, both functions $G_k$ and $G^\mathrm{E}_k$ are continuous in $\mathbb R^2\setminus\Theta$ and coincide on a dense subset of $\mathbb R^2\setminus\Theta$. This completes the proof.
\end{proof}

The Ewald representation of $G_k$ is particularly useful for extracting its logarithmic singularity at $\Theta$. Define
\begin{align*}
    G^S_k(x) &:=  \frac{1}{4\pi}\sum_{n=-\infty}^\infty \mathrm{e}^{\mathrm in\beta} E_{1}(\zeta^2 ((x_1-nL)^2 + x_2^2)) J_0(k|x-nLe_1|) \quad\text{for }x\in\mathbb R^2\setminus\Theta.
\end{align*}
We shall show that $G^S_k$ contains all the logarithmic singularities at $\Theta$.
\begin{proposition}\label{prop:ewald-split}
    Let $k\in\mathbb R$ and $\zeta>0$.
    Then the following assertions hold:
    \begin{enumerate}\renewcommand{\theenumi}{\roman{enumi}}
        \item $G_k - G^S_k$ extends to a real-analytic function in $\mathbb R^2$,
        \item for a given finite subset $\sigma\subset\mathbb Z$, the function $(k,x)\mapsto D_x^\mu (\hat G_k(x) - G^S_k(x))$ is continuous on $((\mathbb R\setminus\Lambda)\cup\Lambda_\sigma)\times \mathbb R^2$ for every multi-index $\mu$.
    \end{enumerate}
\end{proposition}
\begin{proof}
    We first split $G_k - G^S_k = (\Xi_k - G^S_k) + (G_k - \Xi_k)$, where $\Xi_k$ is defined by
    \begin{align*}
        \Xi_k(x) := \frac{1}{4\pi}\sum_{n=-\infty}^\infty \mathrm{e}^{\mathrm in\beta} \sum_{j=0}^\infty \frac{1}{j!}\left(\frac{k}{2\zeta}\right)^{2j} E_{j+1}(\zeta^2 ((x_1-nL)^2 + x_2^2)) \quad\text{for }x\in\mathbb R^2\setminus\Theta.
    \end{align*}
    We shall show that $\Xi_k - G^S_k$ and $G_k - \Xi_k$ are real-analytic in $\mathbb R^2$.

    \noindent (1) $\Xi_k - G^S_k$

    In analogy with the proof of \cref{prop:ident-G-GE}, we see that the defining series of $G^S_k$ and $\Xi_k$ converge absolutely for each $x\in\mathbb R^2\setminus\Theta$. Thus, using the recursion formula
    \begin{align*}
        E_{j+2}(z)
            &=
            \frac{1}{j+1}\bigl(\mathrm e^{-z}-zE_{j+1}(z)\bigr) \quad\text{for all }j=0,1,\ldots \text{ and }z\in\mathbb C \text{ with }\operatorname{Re}[z]>0,
    \end{align*}
    we obtain
    \begin{align}
    \label{eq:Xi-GS}
        \Xi_k(x) - G^S_k(x) &= \frac{1}{4\pi}
        \sum_{n=-\infty}^\infty
        \mathrm e^{\mathrm i n\beta}
        \mathrm e^{-\zeta^2((x_1-nL)^2+x_2^2)}
        \mathcal H_k(\zeta^2((x_1-nL)^2+x_2^2))
    \end{align}
    for all $x\in\mathbb R^2\setminus\Theta$, where $\mathcal H_k:\mathbb C\to\mathbb C$ is an entire function defined by
    \begin{align*}
        \mathcal H_k(z)
        :=
        \sum_{j=1}^\infty \frac{1}{(j!)^2}
        \left(\frac{k}{2\zeta}\right)^{2j}\sum_{\ell=0}^{j-1}(j-1-\ell)!(-z)^\ell.
    \end{align*}
    The assertion follows from the analyticity of each summand and the Weierstrass test. Indeed, we have
    \begin{align*}
        |\mathrm e^{-z}\mathcal H_k(z)| \le (\mathrm e^{k^2/(4\zeta^2)}-1)\frac{\mathrm e^{|z|-\operatorname{Re}z}}{|z|} \quad\text{for all }z\in\mathbb C\setminus\{0\},
    \end{align*}
    and for each compact neighborhood $V\subset \mathbb C^2$ of a point in $\mathbb R^2$, there exists $C>0$ such that
    \begin{align*}
        |\zeta^2((x_1-nL)^2+x_2^2)|-\operatorname{Re}\left[\zeta^2((x_1-nL)^2+x_2^2)\right] 
        \le C \quad\text{for all }x\in V \text{ and }n\in\mathbb Z,
    \end{align*}
    which implies that the summand in \cref{eq:Xi-GS} is $O(n^{-2})$ uniformly in $V$. Thus $\Xi_k-G^S_k$ is real-analytic in $\mathbb R^2$.

    \noindent (2) $G_k-\Xi_k$

    It suffices to show that the series
    \begin{align*}
        \frac{\mathrm i}{4L} \sum_{s\in\mathbb Z\setminus\sigma_\ast(k)} \frac{\mathrm{e}^{\mathrm ik_s x_1} }{\sqrt{k^2 - k_s^2}}
    \biggl[&
    \mathrm{e}^{-\mathrm i \sqrt{k^2 - k_s^2} x_2} \mathrm{erfc}\, \left( \zeta x_2 - \frac{\mathrm i\sqrt{k^2 - k_s^2}}{2\zeta} \right)
    \notag
    \\
    &+
        \mathrm{e}^{+\mathrm i \sqrt{k^2 - k_s^2} x_2} \mathrm{erfc}\, \left(-\zeta x_2 - \frac{\mathrm i\sqrt{k^2 - k_s^2}}{2\zeta} \right)
    \Bigg]
    \end{align*}
    defines a real-analytic function in $\mathbb R^2$. In analogy with (1), using the estimate $|\operatorname{erfc}(w)|\le C|\mathrm e^{-w^2}|/|w|$ for $w\in\mathbb C$ with positive real part, it is straightforward to check that the summand is $O(\mathrm e^{-cs^2})$ uniformly on each compact neighborhood. Thus $G_k-\Xi_k$ is real-analytic in $\mathbb R^2$. 

    We subsequently prove (ii). Define $\mathcal D_\sigma:=(\mathbb R\setminus\Lambda)\cup\Lambda_\sigma$. Using the Ewald representation, we have 
    \begin{align*}
        \hat G_k(x) - G^S_k(x) &= (\Xi_k(x) - G^S_k(x)) + \sum_{s\in\sigma} \mathrm e^{\mathrm ik_s x_1}\Psi(\sqrt{k^2-k_s^2},x_2)
        \\
        &+ \frac{\mathrm i}{4L} \sum_{s\in\mathbb Z\setminus\sigma} \frac{\mathrm{e}^{\mathrm ik_s x_1} }{\sqrt{k^2 - k_s^2}}
    \Biggl[
    \mathrm{e}^{-\mathrm i \sqrt{k^2 - k_s^2} x_2} \mathrm{erfc}\, \left( \zeta x_2 - \frac{\mathrm i\sqrt{k^2 - k_s^2}}{2\zeta} \right)
    \notag
    \\
    &+
        \mathrm{e}^{+\mathrm i \sqrt{k^2 - k_s^2} x_2} \mathrm{erfc}\, \left(-\zeta x_2 - \frac{\mathrm i\sqrt{k^2 - k_s^2}}{2\zeta} \right)
    \Biggr] \quad\text{in }\mathcal D_\sigma\times(\mathbb R^2\setminus\Theta).
    \end{align*}
    
    First, the function
    \[
     F(\kappa,z)
     :=
     e^{-i\kappa z}
     \operatorname{erfc}
     \left(\zeta z-\frac{i\kappa}{2\zeta}\right)
     +
     e^{i\kappa z}
     \operatorname{erfc}
     \left(-\zeta z-\frac{i\kappa}{2\zeta}\right)
     -2
    \]
    is entire and satisfies $F(0,z)=0$. Hence
    $iF(\kappa,z)/(4L\kappa)$ has an entire extension at
    $\kappa=0$, and this extension is precisely $\Psi$.
    Consequently, the finite sum over $s\in\sigma$, together with all
    its $x$-derivatives, is continuous in $(k,x)$.
    
    Next, let $K\subset\mathcal D_\sigma$ be compact and let
    $X\Subset\mathbb C^2$. In analogy with the proof of (i), we can check that the series defining
    $\Xi_k-G_k^S$ and its $x$-derivatives converge absolutely and uniformly on $K\times X$. For the remaining series, the standard large-argument estimate for $\operatorname{erfc}$, applied uniformly on $X$, implies that for every multi-index $\mu$ there exist $C_\mu,c_\mu>0$ such that $\sup_{(k,x)\in K\times X}
     \left|D_x^\mu Q_s(k,x)\right|
     \le C_\mu e^{-c_\mu s^2}$,
    where $Q_s$ denotes the $s$th summand of the last series. Thus this series and all its $x$-derivative series converge absolutely and uniformly on $K\times X$.
    
    It follows that, for every multi-index $\mu$, the function $(k,x)\mapsto
     D_x^\mu(\hat G_k-G_k^S)(x)$ extends continuously to $\mathcal D_\sigma\times\mathbb R^2$.
\end{proof}

\section{Layer potentials and boundary integral operators}
Using the quasi-periodic Green function $G_k$, which is identical to $G^\mathrm{E}_k$ by \cref{prop:ident-G-GE}, we seek a solution to $(P_k)$ in the form of layer potentials.

Let $\eta>0$. A straightforward strategy is to find a function $\varphi\in C_\beta(\Gamma_\mathrm{per})$ such that 
\begin{align}
\label{eq:layerpotential-naive}
    u_k(x) := \int_\Gamma \left(\frac{\partial G_k}{\partial\nu(y)}(x-y) -\mathrm i\eta G_k(x-y) \right) \varphi(y) \mathrm ds(y)
\end{align}
solves $(P_k)$. While the quasi-periodic Green function $G_k$ is defined for all $k\in\mathbb R$, we cannot expect the combined layer potential \cref{eq:layerpotential-naive} to satisfy the radiation condition due to the linearly growing term in \cref{eq:green-all} when $k\in\Lambda$. 

We overcome this difficulty by using the observation from \cref{prop:Gk-cts}. By the definition of $\hat G_k$, it follows that
\begin{align*}
    u_k(x) &= \int_\Gamma \left(\frac{\partial G_k}{\partial\nu(y)}(x-y) -\mathrm i\eta G_k(x-y) \right) \varphi(y) \mathrm ds(y)
    \\
    &= \int_\Gamma \left(\frac{\partial \hat G_k}{\partial\nu(y)}(x-y) -\mathrm i\eta \hat G_k(x-y) \right) \varphi(y) \mathrm ds(y) 
    \\
    &\qquad + \frac{\mathrm i}{2L}\sum_{s\in\sigma\setminus\sigma_\ast(k)} \frac{\mathrm e^{\mathrm ik_s x_1}}{\sqrt{k^2-k_s^2}} \int_\Gamma \left( \frac{\partial}{\partial \nu(y)} - \mathrm i\eta \right)  \mathrm e^{-\mathrm i k_s y_1} \varphi(y)\mathrm ds(y)
\end{align*}
for all $k\in\mathbb R\setminus\Lambda$ and $x\in\Omega_+$. In view of the continuity of $k\mapsto \hat G_k$ at Wood anomaly frequencies (\cref{prop:Gk-cts}), we put
\begin{align*}
    c_s := 
    \begin{cases}
    \displaystyle
        \frac{\mathrm i}{L} \frac{1}{\sqrt{k^2-k_s^2}} \int_\Gamma \left( \frac{\partial}{\partial \nu(y)} - \mathrm i\eta \right)  \mathrm e^{-\mathrm i k_s y_1} \varphi(y)\mathrm ds(y) &\text{if }s\in\sigma\setminus\sigma_\ast(k),
        \\
        0 &\text{otherwise},
    \end{cases}
\end{align*}
so that
\begin{align*}
    u_k(x) = \int_\Gamma \left(\frac{\partial \hat G_k}{\partial\nu(y)}(x-y) -\mathrm i\eta \hat G_k(x-y) \right) \varphi(y) \mathrm ds(y) + \frac{1}{2}\sum_{s\in\sigma} c_s \mathrm e^{\mathrm ik_s x_1}
\end{align*}
for all $k\in\mathbb R\setminus\Lambda$ and $x\in\Omega_+$. We first show that this layer potential can represent a radiating solution to the Helmholtz equation.

\begin{proposition}
    Let $\eta>0$, let $\sigma\subset\mathbb Z$ be a finite subset, let $k\in(\mathbb R\setminus\Lambda) \cup \Lambda_{\sigma}$, and let $\varphi\in L^1(\Gamma)$. 
    Let $c\in\mathcal C_\sigma$ satisfy
    \begin{align*}
        \frac{\mathrm i}{L}\int_\Gamma \left( \frac{\partial}{\partial \nu(y)} - \mathrm i\eta \right)  \mathrm e^{-\mathrm i k_s y_1} \varphi(y)\mathrm ds(y) - c_s\sqrt{k^2-k_s^2} = 0 \quad\text{for all }s\in\sigma.
    \end{align*}
    Then
    \begin{align*} 
        u(x) := \int_\Gamma \left(\frac{\partial}{\partial \nu(y)} - \mathrm i\eta \right) \hat G_k(x-y)\varphi(y) \mathrm ds(y) + \frac{1}{2}\sum_{s\in\sigma} c_s\mathrm e^{\mathrm ik_sx_1} 
    \end{align*}
    satisfies 
    \begin{align*}
        \begin{cases}
            -\varDelta u - k^2 u = 0 & \text{in }\mathbb R^2\setminus\Gamma_\mathrm{per},
            \\
            u(x+Le_1) = u(x)\mathrm e^{\mathrm i\beta}& \text{in }\mathbb R^2\setminus\Gamma_\mathrm{per},
            \\
            u\text{ satisfies the upward and downward radiation conditions}.
        \end{cases}
    \end{align*}
\end{proposition}
\begin{proof}
    By the assumption on $c_s$, we have
    \begin{align*}
        u(x) = \int_\Gamma \left(\frac{\partial}{\partial \nu(y)} - \mathrm i\eta \right)  G_k(x-y)\varphi(y)\mathrm ds(y) + \frac{1}{2}\sum_{s\in \sigma_\ast(k)} c_s\mathrm e^{\mathrm ik_sx_1} \quad\text{for all }x\in \mathbb R^2\setminus\Gamma_\mathrm{per}.
    \end{align*}
    Thus $u$ satisfies $(-\varDelta-k^2)u = 0$ in $\mathbb R^2\setminus\Gamma_\mathrm{per}$ and the quasi-periodic condition. Therefore it suffices to show that $u$ is radiating.

    Let $x\in\mathbb R^2$ be such that $x_2>\max_{y\in\Gamma} y_2$. Then termwise differentiation of $y\mapsto G_k(x-y)$ is justified on $\Gamma$, and thus
    \begin{align*}
        u(x) = \sum_{s\in\mathbb Z} A_s \mathrm e^{\mathrm ik_sx_1 + \mathrm i\sqrt{k^2-k_s^2}x_2}
        +
        \sum_{s\in\sigma_\ast(k)} B_s x_2 \mathrm e^{\mathrm ik_sx_1}
        +
        \frac{1}{2}\sum_{s\in\sigma_\ast(k)} c_s \mathrm e^{\mathrm ik_sx_1} ,
    \end{align*}
    where $A_s$ and $B_s$ are given by
    \begin{align*}
        A_s &= 
        \begin{cases}
            \displaystyle
            \frac{\mathrm i}{2L\sqrt{k^2-k_s^2}} \int_\Gamma \left(\frac{\partial}{\partial \nu(y)} - \mathrm i\eta \right) \mathrm e^{-\mathrm ik_sy_1 - \mathrm i\sqrt{k^2-k_s^2}y_2}\varphi(y)\mathrm ds(y) & s\in \mathbb Z\setminus \sigma_\ast(k),
        \\
            \displaystyle
            \frac{1}{2L}\int_\Gamma \left(\frac{\partial}{\partial \nu(y)} - \mathrm i\eta \right) (y_2\mathrm e^{-\mathrm ik_sy_1}) \varphi(y) \mathrm ds(y) & s\in\sigma_\ast(k),
        \end{cases}
        \\
        B_s &= -\frac{1}{2L}\int_\Gamma \left(\frac{\partial}{\partial \nu(y)} - \mathrm i\eta \right) \mathrm e^{-\mathrm ik_sy_1} \varphi(y)\mathrm ds(y) \quad s\in\sigma_\ast(k).
    \end{align*}
    By the assumption on $c_s$, we have $B_s=0$ for all $s\in\sigma_\ast(k)$. This implies that $u$ is radiating in the upward direction. Analogously, we can check that $u$ is radiating in the downward direction.
\end{proof}

The well-posedness of the boundary integral equation $(Q_k)$ will be established from the unique solvability of a complementary Robin problem.

\begin{lemma}[Kirsch \cite{kirsch1993diffraction}]\label{lemma:robin}
    Let $k\in\mathbb R$ and let $\eta>0$. Then the Robin problem
    \begin{align*}
        \begin{cases}
            -\varDelta u - k^2 u = 0 & \text{in }\Omega_-,
            \\
            u(x+Le_1) = u(x)\mathrm e^{\mathrm i\beta} & \text{in }\Omega_-,
            \\
            \partial_\nu u = \mathrm i\eta u &\text{on }\Gamma_\mathrm{per},
            \\
            u\text{ is radiating in the downward direction} &
        \end{cases}
    \end{align*}
    admits only the trivial solution $u=0$ in $C^2(\Omega_-)\cap C(\overline\Omega_-)$, where the normal derivative $\partial_\nu u$ of a function $u\in C^2(\Omega_-) \cap C(\overline\Omega_-)$ is understood in the sense that
    \begin{align*}
        \nu(x)\cdot \nabla u(x-h\nu(x)) \xrightarrow[h\to 0_+]{} \partial_\nu u \quad\text{uniformly for }x\in\Gamma_\mathrm{per}.
    \end{align*}
\end{lemma}

Using this result, we shall prove \cref{thm:main-1}.

\begin{proof}[Proof of \cref{thm:main-1}]
    Step 1. Let $k\in\mathbb R$ and choose a finite subset $\sigma\subset\mathbb Z$ such that $k\in(\mathbb R\setminus\Lambda)\cup\Lambda_\sigma$. We first claim that $I+A_k:C_\beta(\Gamma_\mathrm{per})\times\mathcal C_\sigma\to C_\beta(\Gamma_\mathrm{per})\times\mathcal C_\sigma$ is invertible. The mapping property and compactness of the linear operator $A_k$ follow from \cref{prop:cts-lipschitz}.
    Therefore it suffices to show that $I+A_k$ is injective to prove the invertibility. Let $(\varphi,c)\in C_\beta(\Gamma_\mathrm{per})\times\mathcal C_\sigma$ be a solution to $(I+A_k) (\varphi,c) = (0,0)$. Define
    \begin{align*}
        w(x):=
        \int_\Gamma
        \left(
            \frac{\partial\hat G_k}{\partial\nu(y)}(x-y)
            -\mathrm i\eta\hat G_k(x-y)
        \right)
        \varphi(y)\mathrm ds(y)
        +
        \frac12\sum_{s\in\sigma}c_s\mathrm e^{\mathrm ik_sx_1} \ \ \text{for }x\in \mathbb R^2\setminus\Gamma_\mathrm{per}.
    \end{align*}
    Then $w$ satisfies
    \begin{align*}
        \begin{cases}
            -\varDelta w - k^2 w = 0 &\text{in }\mathbb R^2\setminus\Gamma_\mathrm{per},
            \\
            w(x+Le_1) = w(x)\mathrm e^{\mathrm i\beta} &\text{in }\mathbb R^2\setminus\Gamma_\mathrm{per},
            \\
            w\text{ is radiating in the upward and downward directions}.&
        \end{cases}
    \end{align*}
    By \cref{lemma:green-regularity}, the standard jump relations hold (cf. \cite{kress2014linear}); thus we obtain that
    \begin{align*}
        w(x+h\nu(x)) \xrightarrow[h\to 0_+]{} 
        \frac12\varphi(x)
        +
        \frac12(D_k\varphi)(x)
        -
        \frac{\mathrm i\eta}{2}(S_k\varphi)(x)
        +
        \frac12(Rc)(x)
        = 0
    \end{align*}
    uniformly for $x\in\Gamma_\mathrm{per}$. Therefore $w$ solves the homogeneous version of $(P_k)$ and hence $w=0$ in $\overline\Omega_+$ by assumption. Analogously, we have
    \begin{align*}
        w(x-h\nu(x))
        \xrightarrow[h\to0_+]{}
        -\frac12\varphi(x)
        +
        \frac12(D_k\varphi)(x)
        -
        \frac{\mathrm i\eta}{2}(S_k\varphi)(x)
        +
        \frac12(Rc)(x)
        =
        -\varphi(x)
    \end{align*}
    and
    \begin{align*}
        -\nu(x)\cdot\nabla w(x-h\nu(x))
        =
        \nu(x)\cdot \left( \nabla w(x+h\nu(x))
        -
        \nabla w(x-h\nu(x)) \right)
        \xrightarrow[h\to0_+]{}
        \mathrm i\eta\varphi(x)
    \end{align*}
    uniformly for $x\in\Gamma_\mathrm{per}$. Thus $w$ solves the Robin problem in \cref{lemma:robin} and hence $w=0$ in $\overline\Omega_-$. In summary, we obtain $\varphi=0$ and hence $\sum_{s\in\sigma} c_s \mathrm e^{\mathrm ik_s x_1} = 0$ in $\mathbb R^2\setminus\Gamma_\mathrm{per}$, i.e., $c=0$. Therefore $I+A_k$ is invertible.

    Step 2. Let $g\in C_\beta(\Gamma_\mathrm{per})$. We claim that $(P_k)$ admits a unique solution with the asserted solution formula. From Step 1, there exists a unique solution $(\varphi,c)$ to $(Q_k)$. Using this solution, we define
    \begin{align*}
        u_k(x) := \int_\Gamma \left(\frac{\partial\hat G_k}{\partial\nu(y)}(x-y) -\mathrm i\eta \hat G_k(x-y) \right) \varphi(y) \mathrm ds(y) + \frac{1}{2}\sum_{s\in\sigma} c_s\mathrm e^{\mathrm ik_sx_1}.
    \end{align*}
    By the same arguments as in Step 1, we can show that $u_k$ belongs to $C^2(\Omega_+)\cap C(\overline\Omega_+)$ with $u_k=g$ on $\Gamma_\mathrm{per}$ and it solves $(P_k)$. It remains to show that this solution is unique. If $v$ is another solution of $(P_k)$ with the same boundary data $g$, then $u_k-v$ solves the homogeneous version of $(P_k)$. By the assumed uniqueness of the homogeneous problem, it follows that $u_k-v=0$.

    Step 3. 
    Since the set of Wood anomaly frequencies is locally finite, we may choose a finite set $\sigma\subset\mathbb Z$ and a neighborhood $U\subset J$ of $k_0$ such that $U\subset(\mathbb R\setminus\Lambda)\cup \Lambda_\sigma$. By \cref{prop:cts-lipschitz}, the mappings $k\mapsto S_k$ and $k\mapsto D_k$ are continuous from $U$ into $\mathcal L(C_\beta(\Gamma_\mathrm{per}),C_\beta(\Gamma_\mathrm{per}))$. Therefore $k\mapsto A_k$ is also continuous from $U$ into $\mathcal L(C_\beta(\Gamma_\mathrm{per})\times\mathcal C_\sigma,C_\beta(\Gamma_\mathrm{per})\times\mathcal C_\sigma)$.

    Step 4. Define $(\varphi_k,c_k)\in C_\beta(\Gamma_\mathrm{per})\times\mathcal C_\sigma$ as the unique solution to $(Q_k)$ with $g=g_k$ for each $k\in U$. By Step 2, $(P_k)$ with $g=g_k$ is uniquely solvable, and its solution $u_k\in C^2(\Omega_+)\cap C(\overline\Omega_+)$ is written as
    \begin{align*}
        u_k(x) = 
        \begin{cases}
            \displaystyle
            \int_\Gamma \left(\frac{\partial\hat G_k}{\partial\nu(y)}(x-y) -\mathrm i\eta \hat G_k(x-y) \right) \varphi_k(y) \mathrm ds(y) + \frac{1}{2}\sum_{s\in\sigma} (c_{k})_s\mathrm e^{\mathrm ik_sx_1} & x\in \Omega_+,
            \\
            \displaystyle
            \frac{1}{2}\varphi_k(x) + \int_\Gamma
        \left(
            \frac{\partial\hat G_k}{\partial\nu(y)}(x-y)
            -\mathrm i\eta\hat G_k(x-y)
        \right)
        \varphi_k(y)\mathrm ds(y)
        +
        \frac12\sum_{s\in\sigma}(c_{k})_s\mathrm e^{\mathrm ik_sx_1} & x\in\Gamma_\mathrm{per}.
        \end{cases}
    \end{align*}
    For every compact subset $V\subset\overline\Omega_+$, this yields
    \begin{align*}
        \sup_{x\in V} |u_\kappa(x)-u_{k_0}(x)| &\leq C\left( \|\varphi_\kappa-\varphi_{k_0}\|_{C_\beta(\Gamma_\mathrm{per})}  +   \|c_\kappa-c_{k_0}\|_{\mathcal C_\sigma} \right)
        \\
        &+ \sup_{x\in V} \| (\partial_{\nu}-\mathrm i\eta) \hat G_\kappa(x-\cdot) \varphi_\kappa(\cdot) - (\partial_{\nu}-\mathrm i\eta) \hat G_{k_0}(x-\cdot) \varphi_{k_0}(\cdot)\|_{L^1(\Gamma)}
        \\
        &\xrightarrow[\kappa\to k_0]{} 0
    \end{align*}
    for some $C>0$, since $k\mapsto \varphi_k$ is continuous from $U$ into $C_\beta(\Gamma_\mathrm{per})$ by Step 3 and
    \begin{align*}
        \sup_{x\in V} \| (\partial_{\nu}-\mathrm i\eta) \hat G_\kappa(x-\cdot)- (\partial_{\nu}-\mathrm i\eta)\hat G_{k_0}(x-\cdot)\|_{L^1(\Gamma)} \xrightarrow[\kappa\to k_0]{} 0
    \end{align*}
    by \cref{prop:Gk-cts} together with
    \begin{align*}
        \hat G_\kappa(x) - \hat G_{k_0}(x)
        &= \left( G_{\kappa}(x) - G_{k_0}(x) - \frac{\mathrm i}{2L}\sum_{s\in\sigma_\ast(k_0)}\frac{\mathrm e^{\mathrm ik_sx_1}}{\sqrt{\kappa^2-k_s^2}} \right) 
        \\
        &+\frac{\mathrm i}{2L} \sum_{s\in\sigma\setminus\sigma_\ast(k_0)} \left( \frac{\mathrm e^{\mathrm ik_sx_1}}{\sqrt{k_0^2-k_s^2}} - \frac{\mathrm e^{\mathrm ik_sx_1}}{\sqrt{\kappa^2-k_s^2}} \right)
        \notag
    \end{align*}
    for all $\kappa\in U\setminus\Lambda$ and $x\in\mathbb R^2\setminus\Theta$.
    This completes the proof.
\end{proof}

\section{Nystr\"om method}
For each frequency $k\in\mathbb R$, \cref{thm:main-1} states that the scattering problem $(P_k)$ can be reduced to the boundary integral equation $(Q_k)$ under the uniqueness assumption. This motivates us to develop a numerical scheme for approximating a solution to $(Q_k)$. 

To this end, we consider the following equivalent formulation of $(Q_k)$ using the parameterization of $\Gamma_\mathrm{per}$ by $\gamma$:
\begin{align*}
    (\widetilde Q_k)\quad&\text{Find }(\widetilde\varphi, \widetilde c)\in C(\mathbb T)\times\mathcal C_\sigma \text{ such that }
    \\
    &\Biggl(
    \begin{bmatrix}
        I & \\ & I
    \end{bmatrix}
    +
    \underbrace{
    \begin{bmatrix}
        \widetilde D_k - \mathrm i\eta \widetilde S_k & \widetilde R
    \\
        \widetilde T & -\mathrm{diag} (1+\sqrt{k^2-k_s^2} )_{s\in\sigma}
    \end{bmatrix}
    }_{=:\widetilde A_k}
    \Biggr)
    \begin{pmatrix}
        \widetilde\varphi \\ \widetilde c
    \end{pmatrix}
    =
    \begin{pmatrix}
        2 \widetilde g_k \\ 0
    \end{pmatrix}
    ,
\end{align*}
where $\widetilde g_k$ is defined by $\widetilde g_k(t):=e^{-\mathrm i\beta t/(2\pi)}g_k(\gamma(t))
$ and $\widetilde S_k,\widetilde D_k:C(\mathbb T)\to C(\mathbb T)$ are the compact linear operators defined by 
\begin{align*}
    (\widetilde S_k\widetilde\varphi)(t) &:= 2\int_0^{2\pi}
        \mathrm e^{-\mathrm i\beta(t-\tau)/(2\pi)}
        \hat G_k(\gamma(t)-\gamma(\tau))
        |\gamma^\prime(\tau)|
        \widetilde\varphi(\tau)
        \mathrm d\tau,
    \\
    (\widetilde D_k\widetilde\varphi)(t)
        &:=
        2\int_0^{2\pi}
        \mathrm e^{-\mathrm i\beta(t-\tau)/(2\pi)}
        \frac{\partial \hat G_k}{\partial\nu(y)}
        (\gamma(t)-\gamma(\tau))
        |\gamma^\prime(\tau)|
        \widetilde\varphi(\tau)
        \mathrm d\tau.
\end{align*}
In addition, the linear operators $\widetilde R:\mathcal C_\sigma\to C(\mathbb T)$ and $\widetilde T:C(\mathbb T)\to\mathcal C_\sigma$ are defined by
\begin{align*}
    (\widetilde R c)(t)
        &:=
        \mathrm e^{-\mathrm i\beta t/(2\pi)}
        \sum_{s\in\sigma}
        c_s
        \mathrm e^{\mathrm ik_s\gamma_1(t)},
    \\
        (\widetilde T\widetilde\varphi)_s
        &:=
        \frac{\mathrm i}{L}
        \int_0^{2\pi}
        (\partial_{\nu(y)} - \mathrm i\eta) \mathrm e^{-\mathrm ik_s y_1}|_{y=\gamma(\tau)}
        \mathrm e^{\mathrm i\beta\tau/(2\pi)}
        |\gamma^\prime(\tau)|
        \widetilde\varphi(\tau)
        \mathrm d\tau.
\end{align*}
It is clear that $(\widetilde Q_k)$ is uniquely solvable if and only if $(Q_k)$ is uniquely solvable. In this case, their unique solutions are related by $\widetilde \varphi(t)=\mathrm e^{-\mathrm i\beta t/(2\pi)} \varphi(\gamma(t))$ for all $t\in\mathbb T$ and $\widetilde c = c$.

\subsection{Approximation by quadrature}

In view of the assumed regularity conditions on $\gamma$ and $g$, the Nystr\"om method is a natural choice for the problem $(\widetilde Q_k)$. The difficulty is, however, that the logarithmic singularity inherited from the single-layer part $\widetilde S_k$ degrades the convergence rate of quadrature rules. The standard approach is to find a splitting $A(t,\tau) = A^{(1)}(t,\tau)\log (4\sin^2((t-\tau)/2)) + A^{(2)}(t,\tau)$ of a given kernel $A$ such that $A^{(1)}$ and $A^{(2)}$ are smooth and periodic. This strategy works well for $A(t,\tau):= H^{(1)}_0(k|\xi(t)-\xi(\tau)|)$ by choosing $A^{(1)}(t,\tau)=(\mathrm i/\pi)J_0(k|\xi(t)-\xi(\tau)|)$, where $\xi:\mathbb T\to\mathbb R^2$ is a parameterization of a smooth Jordan curve in $\mathbb R^2$ \cite{kress1991boundary,kress2014linear}. While the Hankel function cancels out the singularity of the quasi-periodic Green function $G_k$ at a single point in $\Theta$, the same splitting of $H^{(1)}_0(k|\gamma(t)-\gamma(\tau)|)$ does not work since $\Gamma$ is not a Jordan curve and thus $(t,\tau)\mapsto |\gamma(t)-\gamma(\tau)|$ is not periodic in our setting. Indeed, the singularities of the periodic kernel occur when $t-\tau\in2\pi\mathbb Z$, corresponding to the lattice points in $\Theta$, whereas the non-periodic quantity $|\gamma(t)-\gamma(\tau)|$ captures only the local singularity at $t=\tau$.

Instead, we rely on the Ewald representation of $G_k$ derived in Section \ref{ss:ewald}. In particular, by \cref{prop:ewald-split} it suffices to find a suitable quadrature rule for the singular part $G^S_k$. To this end, we define
\begin{align*}
    \mathcal E(\theta)
    :=
    \sum_{\ell\in\mathbb Z}
    \mathrm e^{-\mathrm i\beta(\theta-2\pi\ell)/(2\pi)}
    E_1\!\left(\zeta^2(\theta-2\pi\ell)^2\right)
    \quad \text{for }\theta\in\mathbb R\setminus(2\pi\mathbb Z)
\end{align*}
and consider the quadrature rule for the convolution $(\mathcal E\ast f)(t):=\int_0^{2\pi} \mathcal E(t-\tau)f(\tau)\mathrm d\tau$, where $f$ is a Lipschitz function. After the parametrization $x=\gamma(t)$, $y=\gamma(\tau)$, the logarithmic part of the single-layer contribution generated by $G_k^S$ is represented, up to a continuous remainder, by the periodic convolution kernel $\mathcal E(t-\tau)$.

\begin{lemma}\label{lemma:error-quadrature}
    Let $\zeta>0$. Then for every $f\in C(\mathbb T)$ it holds that
    \begin{align*}
        \sup_{t\in \mathbb T} \left| 
            \int_0^{2\pi} \mathcal E(t-\tau) f(\tau) \mathrm d\tau
            -
            \sum_{j=0}^{2n-1} R^{(n)}_j(t) f(t^{(n)}_j)
        \right| 
        \xrightarrow[n\to\infty]{} 0.
    \end{align*}
    Moreover,
    \begin{align*}
        \sup_{t\in \mathbb T} 
        \left| 
            \int_0^{2\pi} \mathcal E(t-\tau) f(\tau) \mathrm d\tau
            -
            \sum_{j=0}^{2n-1} R^{(n)}_j(t) f(t^{(n)}_j)
        \right| &
        \\
        \le
        2\sqrt{2}(1 + \pi/2) &\pi^{5/4} \sqrt{\zeta^{-1}\coth (\pi^{3/2}\zeta)} \operatorname{Lip}(f) n^{-1}
    \end{align*}
    for all $f\in C^{0,1}(\mathbb T)$ and $n\in\mathbb N$, 
    where $t^{(n)}_j$ and $R^{(n)}_j$ are defined in \cref{eq:def-t,eq:def-R}, respectively.
\end{lemma}
\begin{proof}
    Since $E_1(\zeta^2\theta^2) = -\gamma_\mathrm{E} - \log(\zeta^2\theta^2) + O(\theta^2)$ for small $\theta$, while the remaining terms decay exponentially, we have $\mathcal E\in L^2(\mathbb T)$. 
    We first estimate $\|\mathcal E\|_{L^2(\mathbb T)}$. From Parseval's identity for the Fourier coefficients $\widehat{\mathcal E}_p$ of $\mathcal E$, it follows that
    \begin{align}
        \|\mathcal E\|_{L^2(\mathbb T)}^2 &= 2\pi\sum_{p=-\infty}^\infty | \widehat{\mathcal E}_p |^2 \le 2\pi \sum_{p=-\infty}^\infty \frac{2}{(p+\beta/(2\pi))^2+\pi\zeta^2}
        \notag \\ 
        &= 
        4\pi^{3/2}\zeta^{-1}
        \frac{ \sinh(2\pi^{3/2}\zeta)}{\cosh(2\pi^{3/2}\zeta) - \cos \beta}
        \le 4\pi^{3/2}\zeta^{-1} \coth (\pi^{3/2}\zeta).
        \label{eq:tmp:12}
    \end{align}
    In particular, this implies that $f\mapsto \mathcal E\ast f$ is bounded from $C(\mathbb T)$ into itself.
    
    Let $f\in C(\mathbb T)$. For each $n\in\mathbb N$, we define $P_n f\in C(\mathbb T)$ as the trigonometric interpolant of $f$ with nodes $t^{(n)}_0,\ldots,t^{(n)}_{2n-1}$, given by 
    \begin{align*}
        (P_n f)(t) = \sum_{j=0}^{2n-1} f(t^{(n)}_j) \mathcal L^{(n)}_j(t) \quad\text{for all }t\in\mathbb T
    \end{align*}
    with the Lagrange basis
    \begin{align*}
        \mathcal L^{(n)}_j(t) = \frac{1}{2n}\sum_{p=-n}^{n-1}\mathrm e^{\mathrm i p(t-t^{(n)}_j)} \quad\text{for }j=0,1,\ldots,2n-1\text{ and } t\in\mathbb T.
    \end{align*}
    Then, using the Fourier series, we obtain that $R^{(n)}_j = \mathcal E\ast \mathcal L^{(n)}_j$ for all $n\in\mathbb N$ and $j=0,1,\ldots,2n-1$, which implies 
    \begin{align*}
        \sum_{j=0}^{2n-1} R^{(n)}_j(t) f(t^{(n)}_j) = \int_0^{2\pi} \mathcal E(t-\tau) (P_n f)(\tau)\mathrm d\tau \quad\text{for all }n\in\mathbb N \text{ and }t\in\mathbb T.
    \end{align*}
    Therefore, if $f$ is a trigonometric polynomial, i.e., $P_nf=f$ for sufficiently large $n\in\mathbb N$, then the asserted uniform convergence trivially holds. In addition, it is straightforward to see that $\|P_n f\|_{L^2(\mathbb T)}\le \sqrt{2\pi}\|f\|_{C(\mathbb T)}$, which implies that the family of linear operators
    \begin{align*}
        \left\{ f\mapsto \sum_{j=0}^{2n-1} R^{(n)}_j(\cdot) f(t^{(n)}_j) \right\}_{n\in\mathbb N}
    \end{align*}
    from $C(\mathbb T)$ into itself is uniformly bounded.
    Therefore, by density of trigonometric polynomials in $C(\mathbb T)$, the same uniform convergence holds for all $f\in C(\mathbb T)$.

    Subsequently, assume that $f\in C^{0,1}(\mathbb T)$ and let $n\in\mathbb N$ be fixed. Then 
    \begin{align} 
        \left|
            \int_0^{2\pi} \mathcal E(t-\tau) f(\tau) \mathrm d\tau
            -
            \sum_{j=0}^{2n-1} R^{(n)}_j(t) f(t^{(n)}_j)
            \right|
        &= \left|\int_0^{2\pi} \mathcal E(t-\tau) (f-P_n f)(\tau)\mathrm d\tau \right|
        \notag \\ 
        &\le \| \mathcal E \|_{L^2(\mathbb T)} \| f - P_n f\|_{L^2(\mathbb T)} \quad\text{for all }t\in\mathbb T.
        \label{eq:tmp:11}
    \end{align}
    Thus it suffices to estimate the error $\| f - P_n f\|_{L^2(\mathbb T)}$. To this end, we define the truncated Fourier series
    \begin{align*}
        (S_n f)(t) := \frac{1}{2\pi}\sum_{p=-n}^{n-1}\mathrm e^{\mathrm i p t} \int_0^{2\pi} f(\tau)\mathrm e^{-\mathrm i p\tau}\,\mathrm d\tau \quad\text{for }t\in\mathbb T,
    \end{align*}
    so that $\| f - P_n f\|_{L^2(\mathbb T)} \leq \| f - S_n f\|_{L^2(\mathbb T)} + \| S_n f - P_n f\|_{L^2(\mathbb T)}$. By Parseval's identity, we have
    \begin{align*}
        \| f - S_n f\|_{L^2(\mathbb T)}^2 &= 2\pi\sum_{p\notin \{-n,\ldots,n-1\}} \left| \frac{1}{2\pi} \int_0^{2\pi} f(\tau)\mathrm e^{-\mathrm i p\tau}\,\mathrm d\tau \right|^2
        \\
        &\le \frac{2\pi}{n^2} \sum_{p=-\infty}^\infty \left| \frac{p}{2\pi} \int_0^{2\pi} f(\tau)\mathrm e^{-\mathrm i p\tau}\,\mathrm d\tau \right|^2
        = \frac{1}{n^2} \| f^\prime \|_{L^2(\mathbb T)}^2,
    \end{align*}
    where $f^\prime\in L^\infty(\mathbb T)$ denotes the weak derivative. Observing that
    \begin{align*}
        (P_n f)(t)
        =
        \frac{1}{2\pi}
        \sum_{p=-n}^{n-1}
        \left(
            \sum_{m=-\infty}^\infty
            \int_0^{2\pi} f(\tau)\mathrm e^{-\mathrm i (p+2nm)\tau}\,\mathrm d\tau
        \right)
        \mathrm e^{\mathrm i p t}
        \quad\text{for all }t\in\mathbb T,
    \end{align*}
    we obtain
    \begin{align*}
        \| S_n f - P_n f\|_{L^2(\mathbb T)}^2 &= \frac{1}{2\pi} \sum_{p=-n}^{n-1}
        \left| 
            \sum_{m\neq 0}
            \int_0^{2\pi} f(\tau)\mathrm e^{-\mathrm i (p+2nm)\tau}\,\mathrm d\tau
        \right|^2
        \\
        &\le \frac{\pi^3}{2n^2} \sum_{p=-\infty}^\infty \left| \frac{p}{2\pi} \int_0^{2\pi} f(\tau)\mathrm e^{-\mathrm i p\tau}\,\mathrm d\tau \right|^2
        = \frac{\pi^2}{4n^2} \| f^\prime \|_{L^2(\mathbb T)}^2.
    \end{align*}
    Therefore, we obtain
    \begin{align}
    \label{eq:tmp:13}
        \| f - P_n f\|_{L^2(\mathbb T)} \le (1 + \pi/2) \| f^\prime \|_{L^2(\mathbb T)} n^{-1} \le \sqrt{2\pi}(1 + \pi/2) \operatorname{Lip}(f) n^{-1}.
    \end{align}
    Inserting \cref{eq:tmp:12,eq:tmp:13} into \cref{eq:tmp:11}, we obtain the asserted error estimate.
\end{proof}

In view of the previous lemma, we approximate the linear operator $\widetilde A_k:C(\mathbb T)\times\mathcal C_\sigma\to C(\mathbb T)\times\mathcal C_\sigma$ by the sequence $(\widetilde A_{k,n})$ of linear operators on $C(\mathbb T)\times\mathcal C_\sigma$ defined as
\begin{align*}
    \widetilde A_{k,n}
        =
        \begin{bmatrix}
            \widetilde A^{(1)}_{n} + \widetilde A^{(2)}_{k,n} & \widetilde R\\
            \widetilde T_{n} & -\mathrm{diag} (1+\sqrt{k^2-k_s^2} )_{s\in\sigma}
        \end{bmatrix},
\end{align*}
where the finite-rank operators $\widetilde A^{(1)}_n$, $\widetilde A^{(2)}_{k,n}$, and $\widetilde T_n$ are defined by
\begin{align*}
    (\widetilde A^{(1)}_n\varphi)(t) &:= -\frac{\mathrm i\eta}{2\pi}\sum_{j=0}^{2n-1} R^{(n)}_j(t) |\gamma^\prime(t^{(n)}_j)| \varphi(t^{(n)}_j),
    \\
    (\widetilde A^{(2)}_{k,n}\varphi)(t) &:= \frac{\pi}{n}\sum_{j=0}^{2n-1} H_k(t,t^{(n)}_j) \varphi(t^{(n)}_j),
    \\
    H_k(t,\tau) &:= 
    \biggl[
        2\mathrm e^{-\mathrm i\beta(t-\tau)/(2\pi)}
        (\partial_{\nu(y)}-\mathrm i\eta) \hat G_k(\gamma(t)-y)|_{y=\gamma(\tau)}
    + \frac{\mathrm i\eta}{2\pi}\mathcal E(t-\tau) 
    \biggr]
    |\gamma^\prime(\tau)|,
    \\
    (\widetilde T_n\varphi)_s &:= \frac{\pi\mathrm i}{nL} \sum_{j=0}^{2n-1}
        \left.(\partial_{\nu(y)} - \mathrm i\eta) \mathrm e^{-\mathrm ik_s y_1}\right|_{y=\gamma(t^{(n)}_j)}
        \mathrm e^{\mathrm i\beta t^{(n)}_j/(2\pi)}
        |\gamma^\prime(t^{(n)}_j)|
        \widetilde\varphi(t^{(n)}_j).
\end{align*}

\begin{proposition}\label{prop:holder-cts-Hk}
Let $K$ be a compact subset of $\mathbb R$, let $\eta>0$, and let $\sigma\subset\mathbb Z$ be a finite subset such that $K\subset (\mathbb R\setminus\Lambda)\cup \Lambda_\sigma$. Assume that $\gamma\in C^{2,\alpha}(\mathbb R)$ for some $\alpha\in (0,1]$.
Then the kernel $H_k$ admits a unique continuous extension to $\mathbb T\times\mathbb T$. Moreover, $(k,t,\tau)\mapsto H_k(t,\tau)$ is continuous on $K\times\mathbb T\times\mathbb T$, and $\sup_{k\in K} \|H_k\|_{C^{0,\alpha}(\mathbb T\times\mathbb T)} <\infty$.
\end{proposition}
\begin{proof}
    By \cref{prop:ewald-split} (ii), $(k,x)\mapsto \hat G_k(x)-G_k^S(x)$ and all its $x$-derivatives are continuous on $\bigl((\mathbb R\setminus\Lambda)\cup\Lambda_\sigma\bigr)\times\mathbb R^2$. Consequently, the contribution of $\hat G_k(x)-G_k^S$ to $H_k$ is jointly continuous in $(k,t,\tau)$ and is uniformly bounded in $C^{0,\alpha}(\mathbb T^2)$ for $k\in K$. 

    It remains to analyze the contributions of $G_k^S$ and $\mathcal E$.
    In view of the $2\pi$-periodicity of $H_k$, we first establish the required estimates in a coordinate neighborhood of an arbitrary point of the diagonal $\Delta:=\{(t,t):t\in\mathbb T\}$. Fix $\delta\in(0,\pi)$ and $(t_0,t_0)\in\Delta$, and choose a sufficiently small neighborhood $U$ of $(t_0,t_0)$ together with real lifts of $t$ and $\tau$, still denoted by $t$ and $\tau$, such that $|t-\tau|<\delta$ on $U$ and $t-\tau=0$ on $U\cap\Delta$.
    Define $r(t,\tau)=\gamma(t)-\gamma(\tau)$. In $U$, the only singular terms are the $n=0$ term in $G_k^S(r(t,\tau))$ and the $\ell=0$ term in $\mathcal E(t-\tau)$. We therefore write
    \begin{align*}
      G_k^S(r)
      &=
      \frac{1}{4\pi}
      E_1(\zeta^2|r|^2)J_0(k|r|)
      +
      G_{k,\mathrm{off}}^S(r),
      \\
      E(t-\tau)
      &=
      \mathrm e^{-\mathrm i\beta (t-\tau)/(2\pi)}
      E_1(\zeta^2(t-\tau)^2)
      +
      E_{\mathrm{off}}(t-\tau),
    \end{align*}
    where
    \begin{align*}
      G_{k,\mathrm{off}}^S(r)
      &:=
      \frac{1}{4\pi}
      \sum_{n\in\mathbb Z\setminus\{0\}}
      \mathrm e^{\mathrm in\beta}
      E_1\bigl(
        \zeta^2|r-nLe_1|^2
      \bigr)
      J_0\bigl(
        k|r-nLe_1|
      \bigr),
      \\
      E_{\mathrm{off}}(\theta)
      &:=
      \sum_{\ell\in\mathbb Z\setminus\{0\}}
      \mathrm e^{-\mathrm i\beta(\theta-2\pi\ell)/(2\pi)}
      E_1\bigl(
        \zeta^2(\theta-2\pi\ell)^2
      \bigr).
    \end{align*}
    For sufficiently small $U$, there exists $c_0>0$ such that $|r(t,\tau)-nLe_1| \ge c_0|n|$ and $|t-\tau-2\pi\ell|\ge c_0|\ell|$ for all $n,\ell\in\mathbb Z\setminus\{0\}$ and $(t,\tau)\in U$. The large-argument estimates for $E_1$, together with the
    compactness of $K$, therefore imply that, for every 
    $m\in\mathbb N\cup\{0\}$, there exist constants $C_m,c_m>0$ such that
    \begin{align*}
      &\sup_{\substack{
        k\in K,\,
        (t,\tau)\in U
      }}
      \left|
        D_x^m
        \left[
          E_1\bigl(
            \zeta^2|x-nLe_1|^2
          \bigr)
          J_0\bigl(
            k|x-nLe_1|
          \bigr)
        \right]_{x=r(t,\tau)}
      \right|
      \le
      C_m e^{-c_mn^2},
      \\
      &\sup_{(t,\tau)\in U}
      \left|
        \frac{d^m}{d\theta^m}
        \left[
          \mathrm e^{-\mathrm i\beta(\theta-2\pi\ell)/(2\pi)}
          E_1\bigl(
            \zeta^2(\theta-2\pi\ell)^2
          \bigr)
        \right]_{\theta=t-\tau}
      \right|
      \le
      C_m e^{-c_m\ell^2}.
    \end{align*}
    Hence the two off-diagonal series and all their differentiated series converge uniformly for $(k,t,\tau)\in K\times U$. Their contributions to $H_k$ are therefore continuous in $(k,t,\tau)$ and
    uniformly bounded in $C^{0,\alpha}$ for $k\in K$.
    
    It remains only to analyze the two local singular terms with $n=0$ and $\ell=0$, i.e., 
    \begin{align}
      \mathcal H_k^{\rm sing}(t,\tau)
      &:=
      \frac{
        e^{-i\beta (t-\tau)/(2\pi)}
        |\gamma'(\tau)|
      }{2\pi}
      \Bigl\{
        -\nu(\gamma(\tau))\cdot\nabla_x
        \bigl[
          E_1(\zeta^2|x|^2)J_0(k|x|)
        \bigr]_{x=r(t,\tau)}
      \notag \\ 
      &
        +i\eta
        \bigl[
          E_1(\zeta^2(t-\tau)^2)
          -
          E_1(\zeta^2|r|^2)J_0(k|r(t,\tau)|)
        \bigr]
      \Bigr\}.
      \label{eq:H-sing}
    \end{align}
    
    Define $a(t,\tau) := \int_0^1\gamma'(\tau+\theta (t-\tau))\mathrm d\theta$ for $(t,\tau)\in U$. Then $\inf_{(t,\tau)\in U}|a(t,\tau)|>0$ and $r(t,\tau)=(t-\tau)a(t,\tau)$. Furthermore, since
    $\nu(\gamma(\tau))\cdot\gamma'(\tau)=0$, we have
    \begin{align}
      \frac{\nu(\gamma(\tau))\cdot r(t,\tau)}{|r(t,\tau)|^2}
      =
      \frac{1}{|a(t,\tau)|^2}
      \int_0^1
      (1-\theta)
      \nu(\gamma(\tau))\cdot
      \gamma''(\tau+\theta (t-\tau))\mathrm d\theta.
      \label{eq:normal-ratio}
    \end{align}
    The right-hand side of \cref{eq:normal-ratio} extends to $t=\tau$ with value $\nu(\gamma(\tau))\cdot\gamma''(\tau)/(2|\gamma'(\tau)|^2)$. Since $\gamma\in C^{2,\alpha}$ and $|\gamma'|$ is bounded away from
    zero, this extension belongs to $C^{0,\alpha}$.
    
    A direct differentiation gives
    \begin{align}
      &-\nu(\gamma(\tau))\cdot\nabla_x
      \bigl[
        E_1(\zeta^2|x|^2)J_0(k|x|)
      \bigr]_{x=r(t,\tau)}
      \notag \\ 
      =&
      \frac{\nu(\gamma(\tau))\cdot r(t,\tau)}{|r(t,\tau)|^2}
      \Bigl[
        2\mathrm e^{-\zeta^2|r(t,\tau)|^2}J_0(k|r(t,\tau)|)
        +
        k|r(t,\tau)| J_1(k|r(t,\tau)|)E_1(\zeta^2|r(t,\tau)|^2)
      \Bigr].
      \label{eq:normal-derivative-singular}
    \end{align}
    The first term on the right-hand side of \cref{eq:normal-derivative-singular} is
    $C^{0,\alpha}$ by \cref{eq:normal-ratio}, and it is bounded uniformly for $k\in K$. To analyze the second term, we observe that $(k,\rho)\mapsto k\rho^{-1} J_1(k\rho)$ extends to an entire function on $\mathbb C^2$. In addition, we define 
    \begin{align*}
        F(\rho) :=
        \begin{cases}
            \rho^2 E_1(\zeta^2\rho^2) & \rho>0,
            \\
            0 & \rho=0.
        \end{cases}
    \end{align*}
    Then $F$ is locally Lipschitz since $F^\prime(\rho) = 2\rho E_1(\zeta^2\rho^2) - 2\rho \mathrm e^{-\zeta^2\rho^2} \to 0$ as $\rho\to 0$. Thus the second factor in
    \cref{eq:normal-derivative-singular} extends to a
    $C^{0,\alpha}$ function with uniform estimate in $k\in K$. 
    
    It remains to consider the second term on the right-hand side of \eqref{eq:H-sing}. Define $h(z):=E_1(z)+\gamma_{\rm E}+\log z$, which is analytic at $z=0$. Then we have
    \begin{align*}
      E_1(\zeta^2(t-\tau)^2)
      -
      E_1(\zeta^2|r(t,\tau)|^2)
      =
      \log|a(t,\tau)|^2
      +
      h(\zeta^2(t-\tau)^2)
      -
      h\bigl(
        \zeta^2(t-\tau)^2|a(t,\tau)|^2
      \bigr),
    \end{align*}
    and thus this function belongs to
    $C^{0,\alpha}(U)$. In addition, since $(k,\rho)\mapsto (1-J_0(k\rho))/\rho^2$ is entire, the function $E_1(\zeta^2\rho^2)(1-J_0(k\rho))$ also extends to a uniformly $C^{0,\alpha}$ function.

    Since the diagonal $\Delta$ is compact, it can be covered by finitely many coordinate neighborhoods $U_1,\ldots,U_N$ of the above type. The local extensions constructed on these neighborhoods agree on their overlaps, since they coincide with the original kernel off the diagonal, and the complement of the diagonal is dense. On the compact set away from the diagonal, all the kernels under consideration are continuous in $(k,t,\tau)$ and smooth in $(t,\tau)$, with the relevant derivatives uniformly bounded for $k\in K$. Consequently, these local extensions define a unique continuous extension of $H_k$ to $\mathbb T^2$, jointly continuous on $K\times\mathbb T^2$. Moreover, the finite covering and the uniform local bounds yield the asserted estimate. This completes the proof.
\end{proof}

\begin{lemma}\label{lemma:collective}
    Under the assumptions of \cref{prop:holder-cts-Hk}, the following assertions hold:
    \begin{enumerate}\renewcommand{\theenumi}{\roman{enumi}}
        \item the family $\{ \widetilde A_{k,n} : C(\mathbb T)\times\mathcal C_\sigma\to  C(\mathbb T)\times\mathcal C_\sigma \}_{k\in K,n\in\mathbb N}$ is collectively compact, i.e., the set
        \begin{align*}
            \{ \widetilde A_{k,n}(\varphi,c) : k\in K,\ n\in\mathbb N, \ (\varphi,c)\in C(\mathbb T)\times\mathcal C_\sigma,\ \|\varphi\|_{C(\mathbb T)}+|c|\le 1\}
        \end{align*}
        is relatively compact in $C(\mathbb T)\times\mathcal C_\sigma$,
        \item for each $\varphi\in C(\mathbb T)$ and $c\in\mathcal C_\sigma$, the following $k$-uniform convergence holds:
        \begin{align*}
            \sup_{k\in K} \|(\widetilde A_{k,n}-\widetilde A_k)(\varphi,c)\|_{C(\mathbb T)\times\mathcal C_\sigma} \xrightarrow[n\to\infty]{}0,
        \end{align*}
        \item there exists a constant $C>0$ such that
        \begin{align*}
            \sup_{k\in K} \|(\widetilde A_{k,n}-\widetilde A_k)(\varphi,c)\|_{C(\mathbb T)\times\mathcal C_\sigma}
        \le
        C n^{-\alpha}
        \left(
            \|\varphi\|_{C(\mathbb T)}
            +
            \operatorname{Lip}(\varphi)
        \right)
        \end{align*}
        for all $n\in\mathbb N$, $\varphi\in C^{0,1}(\mathbb T)$, and $c\in\mathcal C_\sigma$.
    \end{enumerate}
\end{lemma}
\begin{proof}
    Assertion (i) is a direct consequence of the Arzel\`a--Ascoli theorem. Indeed, by the proof of \cref{lemma:error-quadrature}, we have the uniform bound
    \begin{align*}
        \sup_{\|\varphi\|_{C(\mathbb T)}\le 1} \| \widetilde A^{(1)}_n\varphi\|_{C(\mathbb T)} \le (2\pi)^{-1/2} \eta \| \mathcal E\|_{L^2(\mathbb T)} \| \gamma^\prime\|_{C(\mathbb T)} \quad\text{for all }n\in\mathbb N
    \end{align*}
    and the equicontinuity estimate
    \begin{align*}
        |(\widetilde A^{(1)}_n\varphi)(t)-(\widetilde A^{(1)}_n\varphi)(\tau)|
        &\le
        (2\pi)^{-1/2}\eta\|\gamma^\prime\|_{C(\mathbb T)}
        \omega_{\mathcal E}(d(t,\tau)) \xrightarrow[t-\tau\to 0]{} 0,
    \end{align*}
    where $d$ is the distance function on $\mathbb T$ and
    \begin{align*}
        \omega_{\mathcal E}(\delta) := \sup_{d(t,\tau)\le \delta} \| \mathcal E(t-\cdot) - \mathcal E(\tau-\cdot) \|_{L^2(\mathbb T)},
    \end{align*}
    which implies the collective compactness of the family $\{ \widetilde A^{(1)}_n\}_{n\in\mathbb N}$. Since the kernel of $\widetilde T$ is continuous on $\mathbb T$, its approximation by the trapezoidal rule, $\widetilde T_n$, forms a collectively compact family. The same argument applies to $\{ \widetilde A^{(2)}_{k,n} \}$. By \cref{prop:holder-cts-Hk}, $(k,t,\tau)\mapsto H_k(t,\tau)$ is continuous on $((\mathbb R\setminus\Lambda)\cup \Lambda_\sigma)\times \mathbb T\times\mathbb T$. Thus the family $\{ \widetilde A^{(2)}_{k,n}\}_{k\in K, n\in\mathbb N}$ is also collectively compact. 

    Assertions (ii) and (iii) follow from \cref{lemma:error-quadrature,prop:holder-cts-Hk} with the convergence rate of the quadrature rule for $C^{0,\alpha}(\mathbb T)$ functions.
\end{proof}

Here we summarize the general theory of collectively compact approximation:
\begin{theorem}[Anselone--Moore \cite{anselone1964approximate}, Kress \cite{kress2014linear}]\label{thm:kress-nystrom}
    Let $X$ be a Banach space, let $K$ be a nonempty compact metric space, let $\{A_k:X\to X\}_{k\in K}$ be a family of compact linear operators such that $I+A_k$ is injective for all $k\in K$ and $k\mapsto A_k$ is continuous in operator norm. Let $\{A_{k,n}:X\to X\}_{k\in K,n\in\mathbb N}$ be a collectively compact family such that for every $\varphi\in X$,
    \begin{align*}
        \sup_{k\in K} \| A_{k,n}\varphi - A_k\varphi \| \xrightarrow[n\to\infty]{} 0.
    \end{align*}
    Let $\{f_k\}_{k\in K}\subset X$.
    Then there exist $N\in \mathbb N$ and $C>0$ such that
    \begin{align*}
        \begin{cases}
            I+A_{k,n} \text{ is invertible}, &
            \\
            \| (I+A_{k,n})^{-1}\| \le C, &
            \\
            \| \varphi_{k}-\varphi_{k,n}\|_X \le C \| (A_k - A_{k,n})\varphi_k\|_X &
        \end{cases}
        \text{ for all }n\ge N \text{ and }k\in K,
    \end{align*}
    where $\varphi_k:=(I+A_k)^{-1}f_k$ and $\varphi_{k,n} := (I+A_{k,n})^{-1}f_k$. Moreover, the condition number $\operatorname{cond}(I+A_{k,n}):=\|I+A_{k,n}\| \|(I+A_{k,n})^{-1}\|$ is bounded uniformly for $k\in K$ and $n\ge N$.
\end{theorem}

\subsection{Truncation of the series in $H_k$}
While the sequence $(\widetilde A_{k,n})$, constructed in the previous subsection, suffices to approximate a solution to $(Q_k)$, we wish to further approximate the operators by truncating the series defining $H_k$ for numerical purposes. This can be done using the locally $k$-uniform exponential decay of the summands. 
\begin{lemma}\label{lemma:truncation}
    Under the assumptions of \cref{prop:holder-cts-Hk}, there exist constants $C>0$ and $\xi>0$ such that 
    \begin{align*}
        \sup_{k\in K} \sup_{t,\tau\in \mathbb T} \left| H_{k}(t,\tau) - H_{k,m}(t,\tau) \right| \le C\mathrm e^{-\xi m} \quad\text{for all }m\in\mathbb N.
    \end{align*}
\end{lemma}
\begin{proof}
All constants below are independent of
$k\in K$, $t,\tau\in\mathbb T$, and $m\in\mathbb N$.
Define $u_\ell:=t-\tau-2\pi\ell$, $x_\ell:=\gamma(t)-\gamma(\tau)-\ell Le_1$, and $\chi_{\ell,m}:=\chi({u_\ell}/(2\pi m))$.
By \cref{prop:ident-G-GE} and the recurrence formula for $E_j$, the kernel $H_k$ is obtained from the formula defining $H_{k,m}$ by replacing
the truncated reciprocal-space sum by the full sum, replacing
$\chi_{\ell,m}$ by $1$, and replacing the finite $j$-sum by the
corresponding infinite sum.

We first consider the reciprocal-space tail. Since $K$ is compact
and $\sigma$ contains every Wood-anomaly index occurring in $K$,
for all sufficiently large $|s|$ we have a constant $c>0$ such that $\mu_s(k):=-\mathrm i\sqrt{k^2-k_s^2} \ge c|s|$ for all $k\in K$ and sufficiently large $|s|$. Define $V:=\{ x-y:x,y\in\Gamma\}$. Using $\operatorname{erfc}(r) \le \mathrm e^{-r^2} / (\sqrt{\pi} r)$ for $r>0$ and $\operatorname{erfc}'(z)=-2e^{-z^2}/\sqrt{\pi}$, we obtain,
for the reciprocal-space summands $q_s(k,x)$ occurring in $W_k$,
\begin{align*}
  \sup_{\substack{k\in K,x\in V}}
  \left(
    |q_s(k,x)|
    +
    |\nabla_xq_s(k,x)|
  \right)
  &\le
  Ce^{-cs^2}.
\end{align*}
Hence
\begin{align}
\label{eq:tmp:58}
  \sup_{\substack{k\in K, x\in V}}
  \left(
    |W_k(x)-W_{k,m}(x)|
    +
    |\nabla_x(W_k-W_{k,m})(x)|
  \right)
  &\le
  Ce^{-cm^2}.
\end{align}

For the real-space terms, define $a_k := k^2/(2\zeta)^2$, $z:=\zeta^2|x|^2$, and 
\begin{align*}
  \mathcal R_k(x)
  &:=
  \int_1^\infty
  t^{-1}
  \exp\left(
    -zt+\frac{a_k}{t}
  \right)\,dt.
\end{align*}
The integral representation gives, uniformly for $k\in K$ and
$|x|\ge1$,
\begin{align}
\label{eq:tmp:59}
  |\mathcal R_k(x)|
  +
  |\nabla_x\mathcal R_k(x)|
  &\le
  Ce^{-c|x|^2}.
\end{align}
The same estimate holds for $E_1(\zeta^2u^2)$ when $|u|\ge1$. Moreover, $|u_\ell| \ge 2\pi|\ell|-2\pi$ and $|x_\ell|\ge L|\ell|-C$. Since $1-\chi_{\ell,m}\ne0$ implies $|u_\ell|>2\pi m$,
\cref{eq:tmp:59} yields
\begin{align}
\label{eq:tmp:510}
  \sum_{\ell\in\mathbb Z}
  |1-\chi_{\ell,m}|
  \left[
    E_1(\zeta^2u_\ell^2)
    +
    |\mathcal R_k(x_\ell)|
    +
    |\nabla_x\mathcal R_k(x_\ell)|
  \right]
  &\le
  Ce^{-cm^2}.
\end{align}

It remains to estimate the truncation in $j$. Let
$\mathcal R_{k,m}$ denote the expression used in $H_{k,m}$, namely
\begin{align*}
  \mathcal R_{k,m}(x)
  &:=
  E_1(z)J_0(k|x|) +
  \mathrm e^{-z}
  \sum_{j=1}^m
  \frac{a_k^j}{(j!)^2}
  \sum_{q=0}^{j-1}
  (j-1-q)!(-z)^q.
\end{align*}
The recurrence formula for $E_j$ gives
\begin{align*}
  \mathcal R_k(x)-\mathcal R_{k,m}(x)
  =
  \sum_{j=m+1}^\infty h_{j,k}(z),
  \quad
  h_{j,k}(z)
  :=
  e^{-z}
  \frac{a_k^j}{(j!)^2}
  \sum_{q=0}^{j-1}
  (j-1-q)!(-z)^q.
\end{align*}
Since $(j-1-q)!/j! \le 1/(q+1)!$, we have
\begin{align*}
  |h_{j,k}(z)|
  +
  \sqrt z\,|\partial_zh_{j,k}(z)|
  &\le
  C\frac{A^j}{j!},
  \qquad z\ge0,
\end{align*}
where $A:=\sup_{k\in K}a_k<\infty$. Consequently,
\begin{align}
\label{eq:tmp:511}
  \sup_{\substack{k\in K,x\in\mathbb R^2}}
  \left(
    |\mathcal R_k(x)-\mathcal R_{k,m}(x)|
    +
    |\nabla_x(\mathcal R_k-\mathcal R_{k,m})(x)|
  \right)
  \le
  C\sum_{j=m+1}^\infty\frac{A^j}{j!}
  \le
  Ce^{-cm}.
\end{align}
There are at most $Cm$ indices $\ell$ for which $\chi_{\ell,m}\ne0$. Thus \cref{eq:tmp:511} implies that the total contribution of the $j$-truncation is bounded by $Cm e^{-cm} \le Ce^{-\xi m}$ for some $\xi>0$.

Combining \cref{eq:tmp:58,eq:tmp:510,eq:tmp:511}, and using the boundedness of
$\gamma'$, $\nu$, and the phase factors, we obtain the asserted estimate.
\end{proof}

We now prove the second result.
\begin{proof}[Proof of \cref{thm:main-2}]
    By \cref{lemma:collective,thm:kress-nystrom}, there exist $N\in\mathbb N$ and $C>0$ such that
    \begin{align}
    \label{eq:main-2:Akn}
        \begin{cases}
            I+\widetilde A_{k,n} \text{ is invertible}, &
            \\
            \| (I+\widetilde A_{k,n})^{-1} \| \le C, &
            \\
            \|\widetilde\varphi_{k,n} - \widetilde\varphi_k\|_{C(\mathbb T)} + |c_k-c_{k,n}| \le C \| (\widetilde A_k - \widetilde A_{k,n})( \widetilde\varphi_k,c_k) \|_{C(\mathbb T)\times \mathcal C_\sigma},&
            \\
            \operatorname{cond}(I+\widetilde A_{k,n}) \le C &
        \end{cases}
    \end{align}
    for all $n\ge N$ and $k\in K$, where $(\widetilde\varphi_k,c_k) := (I+\widetilde A_k)^{-1}(2\widetilde g_k,0)$ and $(\widetilde\varphi_{k,n},c_{k,n}) = (I+\widetilde A_{k,n})^{-1}(2\widetilde g_k,0)$. The mapping property $\widetilde A_k:C^{0,1}(\mathbb T)\times \mathcal C_\sigma\to C^{0,1}(\mathbb T)\times \mathcal C_\sigma$, its compactness, and the continuity of $k\mapsto \widetilde A_k\in \mathcal L(C^{0,1}(\mathbb T)\times\mathcal C_\sigma,C^{0,1}(\mathbb T)\times\mathcal C_\sigma)$ follow from \cref{prop:cts-lipschitz}. Thus we can use the assertion (iii) of \cref{lemma:collective} and obtain
    \begin{align*}
        \sup_{k\in K} \left( \|\widetilde\varphi_{k,n} - \widetilde\varphi_k\|_{C(\mathbb T)} + |c_k-c_{k,n}| \right) \le C n^{-\alpha} \sup_{k\in K}  \left( \| \widetilde\varphi_k \|_{C(\mathbb T)} + \operatorname{Lip}(\widetilde\varphi_k) \right) = C^\prime n^{-\alpha}
    \end{align*}
    for all $n\ge N$, where $C^\prime := C\sup_{k\in K}  \left( \| \widetilde\varphi_k \|_{C(\mathbb T)} + \operatorname{Lip}(\widetilde\varphi_k) \right)<+\infty$. 
    
    Let $\widetilde A_{k,n,m}$ denote the compact operator obtained by replacing the kernel $H_k$ with $H_{k,m}$ in $\widetilde A_{k,n}$. By \cref{lemma:truncation}, there exist $C,\xi>0$ such that
    \begin{align*}
        \sup_{k\in K,n\in\mathbb N} \| \widetilde A_{k,n,m} - \widetilde A_{k,n} \| \le C\mathrm e^{-\xi m} \quad\text{for all }m\in\mathbb N.
    \end{align*}
    Thus, by the standard Neumann series argument, \cref{eq:main-2:Akn} still holds for all $k\in K$, $n\ge N$, and $m\ge M$ even if $\widetilde A_{k,n}$ and $(\widetilde\varphi_{k,n},c_{k,n})$ are replaced by $\widetilde A_{k,n,m}$ and $(\widetilde\varphi_{k,n,m},c_{k,n,m}):=(I+\widetilde A_{k,n,m})^{-1}(2\widetilde g_k,0)$, respectively. Using \cref{lemma:collective} (iii) and $c_{k,n,m} = c_k^{(n,m)}$, we have
    \begin{align*}
        &\| (\varphi_{k}^{(n,m)},c_{k}^{(n,m)}) -  (\varphi_{k},c_{k}) \|_{C(\Gamma_\mathrm{per})\times\mathcal C_\sigma} = \| (\widetilde\varphi_{k,n,m},c_{k,n,m}) -  (\widetilde\varphi_{k},c_{k}) \|_{C(\mathbb T)\times\mathcal C_\sigma} 
        \\
        &\le C
        \left( 
            \| (\widetilde A_{k} - \widetilde A_{k,n} ) (\widetilde\varphi_{k},c_{k}) \|_{C(\mathbb T)\times\mathcal C_\sigma}
            + 
            \|(\widetilde A_{k,n} - \widetilde A_{k,n,m} ) (\widetilde\varphi_{k},c_{k}) \|_{C(\mathbb T)\times\mathcal C_\sigma}
        \right)
        \\
        &\le C( n^{-\alpha} + \mathrm e^{-\xi m})
    \end{align*}
    for all $k\in K$, $n\ge N$, and $m\ge M$.

    By the interpolation formula of the Nystr\"om approximation (cf. \cite[Theorem 12.7]{kress2014linear}), we confirm the equality $\widetilde\varphi_{k,n,m}(t^{(n)}_j) = (\widetilde\varphi_{k}^{(n,m)})_j$ and the existence of uniformly bounded linear operators $R_n:C(\mathbb T)\times\mathcal C_\sigma\to \mathbb C^{2n}\times\mathcal C_\sigma$ and $M_n:\mathbb C^{2n}\times\mathcal C_\sigma \to C(\mathbb T)\times\mathcal C_\sigma$ such that $B_k^{(n,m)} = R_n (I+\widetilde A_{k,n,m}) M_n$ and $(B_k^{(n,m)})^{-1} = R_n (I+\widetilde A_{k,n,m})^{-1} M_n$ (See \cite[\S 14.1]{kress2014linear}). Therefore $\sup_{k\in K,n\ge N,m\ge M} \operatorname{cond}(B_k^{(n,m)}) < +\infty$. This concludes the proof.
\end{proof}

\section{Numerical examples}
We present some numerical examples to illustrate the convergence of the proposed scheme. Throughout this section, we set $L=2\pi$, $\eta=1$, $\zeta=1$, and $m=10$.

\subsection{Flat surface}
We first consider the flat surface given by $\gamma(t) = (Lt/(2\pi),0)$ for $t\in\mathbb R$. In addition, the Dirichlet data is set to $g(x) = \sum_{s=-1}^1 a_s \mathrm e^{\mathrm i s x}$ for all $x\in\mathbb R$ with $a_{-1}:=0.35 - 0.1\mathrm i$, $a_0:=1+ 0.2\mathrm i$, and $a_1:=-0.25 - 0.4\mathrm i$.  Then the assumption of \cref{thm:kirsch} is satisfied, and thus $(Q_k)$ is uniquely solvable by \cref{thm:main-1}. Under the flat-surface setting, the exact solution of $(Q_k)$ is given by
\begin{align*}
    \varphi(x) = \sum_{s=-1}^1 \frac{2\sqrt{k^2-k_s^2} }{\sqrt{k^2-k_s^2} + \eta} a_s \mathrm e^{\mathrm is x}\ \text{for }x\in\mathbb R \quad\text{and}\quad c_s = \frac{2\eta }{\sqrt{k^2-k_s^2} + \eta} a_s \ \text{for }s\in\sigma.
\end{align*}

For fixed $\beta=0$ and $\sigma=\{-1,1,2\}$, we tested the three cases: $k=0.7$ (non-anomalous case), $k=1$ (exact Wood anomaly), and $k=1.0001$ (near Wood anomaly) and numerically calculate the absolute error
\begin{align*}
    \mathrm{error}(n) := \max_{j=0,\ldots,2n-1} |\widetilde \varphi(t^{(n)}_j)-\widetilde{\bm \varphi}^{(n,m)}_j| + |c - c^{(n,m)}|,
\end{align*}
where $(\widetilde\varphi,c)$ and $(\widetilde{\bm \varphi}^{(n,m)},c^{(n,m)})$ are the solutions of $(\widetilde Q_k)$ and $(Q^{(n,m)}_k)$, respectively. 

\begin{figure}
    \centering
    \includegraphics[width=1.0\linewidth]{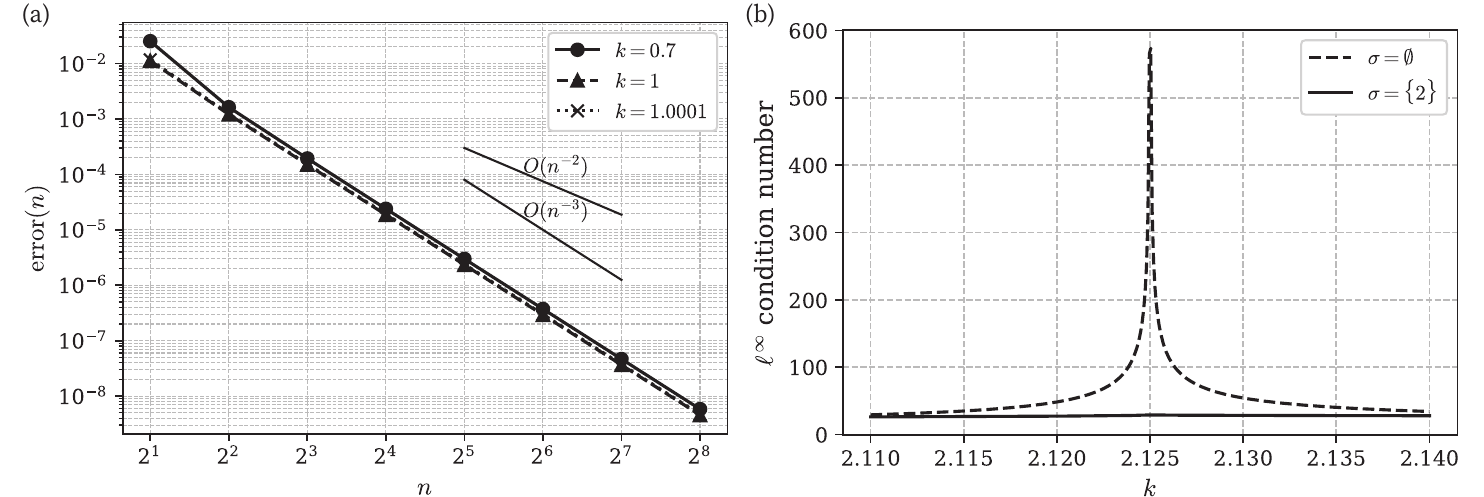}
    \caption{(a) Convergence test for the flat-surface problem. (b) $\ell^\infty$ condition number of $B^{(n,m)}_k$ for varying $k$ around a Wood anomaly frequency $k=2.125$.}
    \label{fig:result}
\end{figure}
\Cref{fig:result} (a) shows the numerical results of the convergence test for each value of $k$. We observe that the numerical error decays at the rate of $O(n^{-3})$, which is better than the theoretical result in \cref{thm:main-2} due to the higher regularity of $g$ and $\gamma$. 

\subsection{Rough surface}
We next consider a rough surface given by the function $\gamma(t)=(t, \pi|\sin(t/2)|^{2.5})$ for $t\in\mathbb R$.
To check the stability of the condition number, we fixed $\beta=\pi/4$ and varied $k$ around a Wood anomaly frequency $k=2.125$ and calculated the $\ell^\infty$ condition number of the matrix $B^{(n,m)}_k$. The result is shown in \cref{fig:result} (b). When $\sigma=\emptyset$ is chosen, i.e., the conventional integral equation is used, the condition number blows up at the anomaly. In the case of $\sigma=\{2\}$, the result indicates a bound on the condition number in the frequency interval, which is consistent with \cref{thm:main-2}.

\section*{Acknowledgments}
We would like to acknowledge K. Morishita for helpful advice and discussion. 
K.M. and H.I. are supported by JSPS KAKENHI Grant Number JP23K28103. K.M. is supported by JSPS KAKENHI Grant Number JP24K17191 and Hirose Foundation.
This work is supported by ‘Joint Usage/Research Center for Interdisciplinary Large-scale Information Infrastructures (JHPCN)’ and ‘High Performance Computing Infrastructure (HPCI)’ in Japan (Project ID: jh260048).


\end{document}